\RequirePackage{rotating}
\documentclass{amsart}
\usepackage{amscd}
\usepackage{amssymb}
\usepackage{rotating,times}

\usepackage{color}
\usepackage{booktabs}
\usepackage{bm}

\usepackage[colorlinks=true, pdfstartview=FitV, linkcolor=blue, pagebackref, citecolor=blue, urlcolor=blue]{hyperref}

\makeatletter
\@namedef{subjclassname@2020}{\textup{2020} Mathematics Subject Classification}
\makeatother

\newtheorem{thm}[equation]{Theorem}
\newtheorem{lem}[equation]{Lemma}

\newtheorem{cor}[equation]{Corollary}
\newtheorem{prop}[equation]{Proposition}
\newtheorem{conj}[equation]{Conjecture}

\newtheorem*{thm*}{Theorem}
\newtheorem*{prop*}{Proposition}
\newtheorem*{cor*}{Corollary}
\newtheorem*{lem*}{Lemma}
\newtheorem*{MT*}{Main Theorem}

\newtheorem*{ques*}{Question}
\newtheorem{conjecture}[equation]{Conjecture}

\theoremstyle{definition} %
\newtheorem{defn}[equation]{Definition}
\newtheorem*{defn*}{Definition}

\newtheorem{eg}[equation]{Example}

\theoremstyle{remark} %
\newtheorem{rmk}[equation]{Remark}

\newtheorem*{rmk*}{Remark}
\newtheorem*{rmks*}{Remarks}

\newcommand{\red}{\mathrm{red}}
\newcommand{\lieE}{\mathfrak{e}}
\newcommand{\lieF}{\mathfrak{f}}

\DeclareMathOperator{\rank}{rank}
\DeclareMathOperator{\car}{char}
\DeclareMathOperator{\Lie}{Lie}
\DeclareMathOperator{\Aut}{Aut}
\DeclareMathOperator{\Hom}{Hom}

\newcommand{\galpha}{\bm{\alpha}}

\newcommand{\domega}{\mathrm{d}\omega}

\DeclareMathOperator{\im}{im}
\DeclareMathOperator{\Trans}{Trans}

\newcommand{\Aff}{\mathbb{A}}

\renewcommand{\P}{\mathbb{P}}
\newcommand{\C}{\mathbb{C}}
\newcommand{\Z}{\mathbb{Z}}

\newcommand{\gbar}{\overline{g}}

\newcommand{\lsub}{\mathfrak{h}}

\newcommand{\hst}{\tilde{\alpha}}

\newcommand{\qform}[1]{{\left\langle{#1}\right\rangle}}                   
\DeclareMathOperator{\PGL}{PGL}
\DeclareMathOperator{\SL}{SL}
\DeclareMathOperator{\Sp}{Sp}
\DeclareMathOperator{\HSpin}{HSpin}

\DeclareMathOperator{\SO}{SO}
\DeclareMathOperator{\PSp}{PSp}

\DeclareMathOperator{\GSp}{GSp}
\newcommand{\gsp}{\mathfrak{gsp}}

\DeclareMathOperator{\GL}{GL}
\DeclareMathOperator{\Ad}{Ad}

\DeclareMathOperator{\Spin}{Spin}

\newcommand{\F}{\mathbb{F}}

\newcommand{\g}{\mathfrak{g}}
\newcommand{\gt}{\tilde{\g}}
\newcommand{\gb}{\bar{\g}}

\newcommand{\n}{\mathfrak{n}}
\newcommand{\gl}{\mathfrak{gl}}
\renewcommand{\sl}{\mathfrak{sl}}

\newcommand{\pgl}{\mathfrak{pgl}}

\newcommand{\so}{\mathfrak{so}}

\newcommand{\tor}{\mathfrak{t}}

\DeclareMathOperator{\Spec}{Spec}

\newcommand{\z}{\mathfrak{z}}

\newcommand{\Gm}{\mathbb{G}_m}
\newcommand{\Ga}{\mathbb{G}_a}

\newcommand{\drho}{\mathrm{d}\rho}

\newcommand{\dpi}{\mathrm{d}\pi}

\newcommand{\la}{\lambda}
\newcommand{\ot}{\otimes}

\newcommand{\eps}{\varepsilon}

\newcommand{\stcolvec}[2]{\left( \begin{smallmatrix} #1 \\ #2 \end{smallmatrix} \right) }

\newcommand{\stbtmat}[4]{\left( \begin{smallmatrix} #1&#2 \\ #3&#4 \end{smallmatrix} \right)}

\newcommand{\Gt}{\widetilde{G}}
\newcommand{\Tt}{\widetilde{T}}
\newcommand{\Tb}{\overline{T}}
\newcommand{\Sb}{\overline{S}}

\newcommand{\Gb}{\bar{G}}

\DeclareMathOperator{\Alt}{Alt}

\DeclareMathOperator{\PSU}{PSU}
\DeclareMathOperator{\PSL}{PSL}

\newcommand{\kx}{k^\times}

\newcommand{\sgp}{G_*}

\newcommand\attopp[2]{\genfrac{}{}{0pt}{}{#1}{#2}}

\newcommand\esevenrt[7]{\textstyle{\attopp{
\hbox to4pt{$\hfil\scriptstyle{#1}\hfil$}
\hbox to4pt{$\hfil\scriptstyle{#3}\hfil$}
\hbox to4pt{$\hfil\scriptstyle{#4}\hfil$}
\hbox to4pt{$\hfil\scriptstyle{#5}\hfil$}
\hbox to4pt{$\hfil\scriptstyle{#6}\hfil$}
\hbox to4pt{$\hfil\scriptstyle{#7}\hfil$}}
{\raise2pt\hbox to4pt{$\hfil\scriptstyle{#2}\hfil$}
\raise2pt\hbox to4pt{$\hfil\scriptstyle{\phantom{0}}\hfil$}}}}
\newcommand\eeightrt[8]{\textstyle{\attopp{
\hbox to4pt{$\hfil\scriptstyle{#1}\hfil$}
\hbox to4pt{$\hfil\scriptstyle{#3}\hfil$}
\hbox to4pt{$\hfil\scriptstyle{#4}\hfil$}
\hbox to4pt{$\hfil\scriptstyle{#5}\hfil$}
\hbox to4pt{$\hfil\scriptstyle{#6}\hfil$}
\hbox to4pt{$\hfil\scriptstyle{#7}\hfil$}
\hbox to4pt{$\hfil\scriptstyle{#8}\hfil$}}
{\raise2pt\hbox to4pt{$\hfil\scriptstyle{#2}\hfil$}
\raise2pt\hbox to4pt{$\hfil\scriptstyle{\phantom{0}}\hfil$}
\raise2pt\hbox to4pt{$\hfil\scriptstyle{\phantom{0}}\hfil$}}}}

\numberwithin{equation}{section}

\begin{document}

\title{Generic stabilizers for simple algebraic groups}

\author[S. Garibaldi]{Skip Garibaldi}
\author[R.M. Guralnick]{Robert M. Guralnick}

\thanks{The second author (RG) was partially supported the NSF grant DMS-1901595 and a Simons Foundation Fellowship 609771.}

\begin{abstract}
We prove a myriad of results related to the stabilizer in an algebraic group $G$ of a generic vector in a representation $V$ of $G$ over an algebraically closed field $k$.  Our results are on the level of group schemes, which carries more information than considering both the Lie algebra of $G$ and the group $G(k)$ of $k$-points.  For $G$ simple and $V$ faithful and irreducible, we prove the existence of a stabilizer in general position, sometimes called a principal orbit type.  We determine those $G$ and $V$ for which the stabilizer in general position is smooth, or $\dim V/G < \dim G$, or there is a $v \in V$ whose stabilizer in $G$ is trivial.
\end{abstract}
%


\maketitle

\setcounter{tocdepth}{1}
\tableofcontents

\section{Introduction}

The aim of this paper is to prove, for an algebraically closed field $k$ of arbitrary characteristic, analogs of results that are known  in the case $k = \C$ concerning an irreducible representation $V$ of a simple linear algebraic group $G$.

The first such result concerns the existence of a \emph{stabilizer in general position} (s.g.p.) also known as a \emph{principal orbit type} for $G$ acting on $V$.  For $v \in V$, write $G_v$ for the fixer of $v$ in $G$; it is a closed sub-group-scheme of $G$.  One says the action of $G$ on $V$ has an s.g.p.\ if there is a closed sub-group-scheme $G_*$ of $G$ and a dense open subset $U$ of $V$ such that for all $u \in U(k)$, there is a $g \in G(k)$ such that $g G_u g^{-1} = G_*$.  We prove in \S\S \ref{sgp.sec}--\ref{sgp.pf}:

\begin{thm} \label{sgp}
Every irreducible representation $V$ of a simple algebraic group $G$ over an algebraically closed field $k$ has an s.g.p.  If $\dim V > \dim G$, the s.g.p.\ $G_*$ is a finite group scheme.  If $\dim V > \dim G$ and $\ker [G \to \GL(V)]$ is central in $G$, then the identity component of $G_*$ is contained in a torus and $\Lie(G_*)$ is a toral subalgebra of $\Lie(G)$.
\end{thm}

In case $\car k = 0$, Richardson proved the existence of an s.g.p.\ under weaker hypotheses, e.g., in case $G$ is merely assumed to be reductive \cite[Th.~A]{Richardson:prin}.  However, that claim can fail when $\car k \ne 0$, see Example \ref{no.sgp}.  (There are additional, related results for $\C$ that do not hold for $k$ of prime characteristic, see Example \ref{not.cofree}.)  Because of this, it is no surprise that the arguments used here when $\car k \ne 0$ are of a fundamentally different nature.  We rely on recent results proved in  \cite{GurLawther},  \cite{GG:large}, \cite{GG:irred}, and \cite{GG:special}, see \S\ref{summ.sec} for a summary.

The previous work showed that, apart from an explicit list of cases, the stabilizer $G_v$ of a generic $v \in V$ is the trivial group scheme, in which case the s.g.p.\ trivially exists.  The proof of Theorem \ref{sgp} involves analyzing the many remaining cases.    Along the way, we determine the s.g.p.\ as a group scheme in almost all cases.

We also prove a result about when $G_*$ is commutative, see \S\ref{commutative.sec}.

We mention that when an s.g.p.\ $G_*$ exists, one obtains as a consequence that the natural map in fppf cohomology $H^1_{\mathrm{fppf}}(k, N_G(G_*)) \to H^1_{\mathrm{fppf}}(k, G)$ is surjective, see \cite[Cor.~4.5]{LoetscherMacD}.  This provides in turn an upper bound on the essential dimension of $G$.   We do not pursue this avenue here.

\subsection*{Smoothness} Another feature that appears when $\car k$ is prime is that the group scheme $G_v$ need not be smooth.  We call out those cases where it happens in the following result, proved in \S\ref{smooth.sec}.  In the statement, the expression ``for generic $v \in V$" means that there is a dense open subset $U$ of $V$ such that the statement holds for all $v \in U(k)$ (in this case, that $G_v$ is smooth).
 
\begin{thm} \label{stab.smooth}
Let $V$ be a faithful and irreducible representation of a simple algebraic group $G$ over an algebraically closed field $k$.  If it is not the case that $G_v$ is smooth for generic $v \in V$, then up to graph automorphism $(G, \car k, V)$ appears in Table \ref{meta.big} or $(G, \car k, V) = (G_2, 2, L(\omega_2))$.
\end{thm}

We were surprised to find that there was an example with $V/G = \Spec k$ yet the generic stabilizer is non-smooth, namely the representation of $G_2$ mentioned in theorem, see Lemma \ref{G2} below.

One could weaken the hypothesis ``faithful'' in the theorem.  Let $N$ be the kernel of $G \to \GL(V)$.  The quotient $G/N$ acts faithfully on $V$ (see \S\ref{back.sec}) and is simple (Lemma \ref{simple.quo}), so Theorem \ref{stab.smooth} applies to it.  If $(G/N)_v$ is smooth (as given by Theorem \ref{stab.smooth}) and $N$ is smooth (a hypothesis to replace ``faithful''), then $G_v$ is smooth because $G_v/N = (G/N)_v$ \cite[Prop.~1.62]{Milne:AG}.

\subsection*{Rings of invariants} For a representation $V$ of $G$, the quotient $V/G$ in the sense of Rosenlicht is defined to be $\Spec k[V]^G$, where $k[V]^G$ is the ring of $G$-invariant functions on $V$.  In case $G$ is reductive, $k[V]^G$ is a finitely generated $k$-algebra \cite[Th.~2]{Sesh:GR}, and it has dimension $\dim V - \dim G + \dim G_v$ for generic $v \in V$.  

Combining the determination of $G_v(k)$ from \cite{GurLawther} with information about the possibilities for $\dim V$ from \cite{luebeck}, we can determine all cases where $\dim V/G$ is ``small''.  In the following result, $L(\hst)$ is the irreducible representation of $G$ with highest weight the highest root $\hst$.  This result is proved in \S\ref{few.sec}.

\begin{thm} \label{kVG}  
Suppose $V$ is a faithful and irreducible representation of a simple algebraic group $G$ over an algebraically closed field $k$.  If $\dim k[V]^G < \dim G$, then $V = L(\hst)$ or $(G, \car k, V)$ belongs to Table \ref{meta.small}, \ref{meta.big}, or \ref{meta.few}, up to graph automorphism.
\end{thm}

\subsection*{Regular orbits}
We also consider when a simple algebraic group acting on an irreducible
module has a regular orbit, i.e., when there exists a vector whose stabilizer is trivial.  Note that a necessary condition for there to exist a regular
orbit is that the stabilizer of a generic vector is finite, for otherwise the dimension of any stabilizer is positive.  In characteristic zero, it turns out that this is also sufficient.  The following result is proved in \S\ref{sec:regular}.

\begin{thm} \label{t:regorbit}   
Suppose that $V$ is a faithful and irreducible representation of a simple algebraic group $G$ over a field $k$ of characteristic $p \ne 2, 3, 5$.  Then exactly one of the following possibilities occurs:
\begin{enumerate}
\renewcommand{\theenumi}{\alph{enumi}}
\item \label{reg.reg} there is some $v \in V$ such that the stabilizer $G_v$ is the trivial group scheme;  
\item \label{reg.dim} $\dim V \le \dim G$;
\item \label{reg.SL4} $p$ is odd, $G = \SL_4$, and up to graph automorphism $V = L(\omega_1 + p^e \omega_2)$ for some $e \ge 1$; or
\item  \label{reg.usual} $p \ne 0$,  $G$ is a quotient of $\SL_{\ell +1}$, and up to graph automorphism $V=L(\omega_1 + p^e \omega_1)$ or $L(\omega_1 + p^e \omega_\ell)$ for some $e \ge 1$.
\end{enumerate} 
\end{thm}

Note that $\dim V \ge \dim G$ is an obvious necessary condition for the existence of a regular orbit.  The theorem says that the condition is also sufficient, apart from a few exceptions, namely the cases $\dim V = \dim G$, \eqref{reg.SL4}, and \eqref{reg.usual}.

\subsection*{Other results}
In addition to the results described so far, we also provide some other applications, such as a result relating the generic stabilizer for $G$ on $V$ with the generic stabilizer for $G$ on a section of $V$ (see \S\ref{section.sec}) and a shorter proof of the classification of groups with the same invariants from \cite{GG:simple} (see Theorem \ref{same.thm}).

\subsection*{Acknowledgements}
The results in this paper weave together and rest on several recent papers, including \cite{GurLawther}.  Although Ross Lawther is not listed here as a co-author, this work would not have been possible without his contributions.  We also thank David Stewart and a referee for their valuable comments on an earlier version of this article.

\section{Notation and background}  \label{back.sec}
Throughout this paper, we assume that $k$ is an algebraically closed field.  We consider algebraic groups over $k$ in the sense of \cite{Milne:AG}, i.e., as affine group schemes of finite type over $k$.  Sometimes we write group scheme when it seems important to do so for clarity of exposition.  

For an algebraic group $G$ and any commutative $k$-algebra $R$, we put $G(R) := \Hom_k(k[G], R)$, the set of $R$-points of $G$; it is an ``abstract'' or ``ordinary'' group.  The algebraic group $G$ is \emph{finite} if it is finite as a scheme over $k$, which holds if and only if $G(k)$ is a finite group.  We say that $G$ is \emph{commutative} if $G(R)$ is abelian for every $k$-algebra $R$.

We also consider the Lie algebra of $G$, which we denote by $\Lie(G)$ or $\g$.  Note that $G$ is \emph{smooth} if and only if $\dim G = \dim \g$, that $G$ is \emph{\'etale} if and only if $\g = 0$, and $G$ is the trivial group scheme $\Spec k$ if and only if $\g = 0$ and $G(k) = 1$.  We put $G^\circ$ for the connected component of the identity in $G$.  It is itself an algebraic group, and $\Lie(G^\circ) = \Lie(G)$.

The algebraic group $G$ is \emph{semisimple} if it is smooth and connected and has no smooth connected solvable normal subgroups other than $1$.  It is \emph{simple} (Milne says ``almost-simple'') if it is semisimple and every proper normal algebraic subgroup is finite.

We say that an action $\rho \!: G \to \Aut(V)$ is  \emph{faithful} if $\ker \rho$ is the trivial group scheme.  Our main results are for a faithful and irreducible representation $V$ of a simple algebraic group $G$.  We view the hypothesis ``faithful'' as harmless.  Indeed,
suppose  $N$ is a normal algebraic subgroup of $G$ that is contained in the kernel of $\rho$.  Then there is a quotient map $q \!: G \to G/N$ \cite[Th.~5.14]{Milne:AG} and a unique morphism $\bar{\rho} \!: G/N \to \Aut(V)$ so that $\rho = q \bar{\rho}$ and $\ker \bar{\rho} = (\ker \rho)/N$ \cite[Th.~5.13, Th.~5.39]{Milne:AG}.  In particular, the induced action $G/\ker \rho \to \Aut(V)$ is faithful.  Moreover, when $G$ is simple, so is $G/N$:

\begin{lem} \label{simple.quo}
If $N$ is a proper normal sub-group-scheme of a simple algebraic group $G$, then $G/N$ is also simple.
\end{lem}

\begin{proof}
The quotient $G/N$ is smooth and connected because $G$ is \cite[Cor.~5.26, Prop.~5.59]{Milne:AG}.  The inverse image $H$ in $G$ of a proper normal sub-group-scheme of $G/N$ is a proper normal subgroup of $G$ containing $N$, so $H$ is finite and its image $H/N$ in $G/N$ is finite.
\end{proof}

In the setting of group schemes, $N$ can be ``large''.  For example,
suppose a representation $\rho \!: G \to \GL(V)$ is obtained as a Frobenius twist of another representation.  Then $\rho$ is not faithful, because $\drho \!: \g \to \gl(V)$ is zero;  $\ker \rho$ contains the first Frobenius kernel $G_1$.  In this case, $G/G_1$ is isomorphic to $G$  \cite[I.9.5]{Jantzen}.

Near the end of this section, we provide another example, where $G = \Sp_{2\ell}$, $\car k = 2$, and $\rho$ is the spin representation.

\subsection*{Stabilizers} For a representation $\rho \!: G \to \GL(V)$ and $v \in V$, we write $G_v$ for the closed sub-group-scheme of $G$ with $R$-points
$G_v(R) = \{ g \in G(R) \mid \rho(g)v = v \}$ for every commutative $k$-algebra $R$, the stabilizer of $v$ in the ordinary group $G(R)$.  The Lie algebra $\Lie(G_v)$ is the annihilator of $v$ in $\g$, denoted $\g_v$:
\[
\g_v = \{ x \in \g \mid \drho(x)v = 0 \}.
\]

The following is a well-known example in the special case $k = \C$, see for example  \cite[23.1]{Gurevich} or \cite[\S1]{AndreevPopov}.

\begin{eg}[binary cubics] \label{char3}
Take $G = \SL_2$ and $V$ the space of binary cubics, i.e., homogeneous polynomials of degree 3 in variables $x$, $y$, over a field $k$ of characteristic $\ne 2$.  A generic vector $v \in V$ is one that vanishes on three distinct lines in $k^2$, such as $v = xy(x-y)$.  One computes that $G_v$ is the group scheme $\Z/3$.  (Compare the case $\la = 3$ in Example \ref{A1} below.)

The element $w := x^2 y$ is not in the orbit of $v$.  One can compute directly that $G_w = 1$, i.e., the $G$-orbit of $w$ is regular.  (Compare Theorem \ref{t:regorbit} in the introduction.)

Focus now on the special case where $\car k = 3$.  Then $V$ is reducible with socle $L(3)$, consisting of cubes of linear forms, and head $L(1)$ spanned by the images of $x^2 y$, $xy^2$.  The semisimplification $V'$ of $V$ is isomorphic to $L(1) \oplus L(1)^{[3]}$.  The stabilizer in $G$ of a generic vector in the natural module $L(1)$ is a 1-dimensional unipotent subgroup $U$, so the stabilizer in $G(k)$ of a generic $v' \in V'$ is an intersection $U(k) \cap U(k)^g$ for a generic $g \in G(k)$, i.e., $G_v(k) = 1$.  On the other hand, the Lie algebra $\g$ acts trivially on $L(1)^{[3]}$, so $\g_{v'} = \g_v = \Lie(U)$.  In summary, $G_{v'}$ is infinitesimal whereas $G_v$ is \'etale.
\end{eg}

\begin{eg}[Diagonalizable groups]
If $G$ is a diagonalizable group scheme, then the stabilizer of every generic $v \in V$ is the kernel of $G \to \GL(V)$.  Therefore the s.g.p.\ exists trivially.

Suppose now merely that the identity component of $G$ is a torus.  In this case the s.g.p.\ need not be the kernel of the action.  (This contradicts \cite[\S7.2, Prop.]{PoV}, whose statement is surely a typo.)  Take for example $V$ to be the natural representation of $\GL_2$ and $G$ to be the normalizer of the diagonal matrices, so $G \cong \Gm \wr \Z/2$.  A  vector $\stcolvec{x}{y} \in V$ with $xy \ne 0$ has stabilizer $\Z/2$ in $G$ with non-identity element $\stbtmat{0}{x/y}{y/x}{0}$.  Choosing another vector in $V$ of the same type and not in the span of the first, we find a stabilizer that is distinct from but $G$-conjugate to the stabilizer of $\stcolvec{x}{y}$.  In this case the s.g.p.\ exists, is not the kernel of the action, and is only determined by $(G, V)$ up to conjugacy class.
\end{eg}

\subsection*{Irreducible representations}
Recall that every irreducible representation $V$ of $G$ has a highest weight $\la$.  Write $\la$ as a sum $\la = \sum_\omega c_\omega \omega$ where the sum runs over the fundamental dominant weights $\omega$.  One says that $\la$ is \emph{restricted} when $p := \car k \ne 0$ if $0 \le c_\omega < p$ for all $\omega$. (In case $\car k = 0$, all dominant weights are, by definition, restricted.)  

Suppose now that $p \ne 0$.  Write $\la = \la_0 + p^r \la_1$ for some $r \ge 1$, where $\la_0 = \sum_\omega c_\omega \omega$ and $0 \le c_\omega < p^r$ for all $\omega$.  If $\la_0$ and $p^{r-1} \la_1$ belong to $T^*$ (e.g., if $G$ is simply connected), then $L(\la) \cong L(\la_0) \ot L(p^{r-1} \la_1)^{[p]}$ \cite[II.3.16]{Jantzen}, the tensor product of $L(\la_0)$ and a Frobenius twist of $L(p^{r-1} \la_1)$.  As a representation of $\g$ (forgetting about the action of $G(k)$), this is the direct sum of $\dim L(\la_1)$ copies of $L(\la_0)$.

We label the simple roots of $G$ as in Table \ref{dynks.table}, which agrees with \cite{Bou:g4} as well as \cite{GurLawther}.  Note that our other references, \cite{GG:large}, \cite{GG:irred}, and \cite{GG:special} follow the numbering of \cite{luebeck}, which is different.

\begin{table}[bth]
\begin{center}
\renewcommand{\arraystretch}{1.7}
\begin{tabular}[c]{ccp{1.4in}} \toprule
name&torsion primes&diagram \\ \midrule
$A_\ell$ ($\ell \ge 1$)&none&
\begin{picture}(7,2)(0,0)
\put(0,1){\circle*{3}}
\put(0,1){\line(1,0){20}}
\put(20,1){\circle*{3}}
\put(20,1){\line(1,0){20}}
\put(40,1){\circle*{3}}
\put(45,-1.6){\mbox{$\cdots$}}
\put(62,1){\circle*{3}}
\put(62,1){\line(1,0){20}}
\put(82,1){\circle*{3}}
\put(82,1){\line(1,0){20}}
\put(102,1){\circle*{3}}

\put(-2,-7){\mbox{\tiny $1$}}
\put(18,-7){\mbox{\tiny $2$}}
\put(38,-7){\mbox{\tiny $3$}}
\put(54,-7){\mbox{\tiny $\ell$$-$$2$}}
\put(75,-7){\mbox{\tiny $\ell$$-$$1$}}
\put(100,-7){\mbox{\tiny $\ell$}}
\end{picture}
\\ 
$B_\ell$ ($\ell \ge 3$)&2&
\begin{picture}(7,2)(0,0)
\put(0,1){\circle*{3}}
\put(0,1){\line(1,0){20}}
\put(20,1){\circle*{3}}
\put(20,1){\line(1,0){20}}
\put(40,1){\circle*{3}}
\put(43,-1.6){ \mbox{$\cdots$}}
\put(62,1){\circle*{3}}
\put(62,1){\line(1,0){20}}
\put(82,1){\circle*{3}}
\put(82,2){\line(1,0){20}}
\put(82,0){\line(1,0){20}}
\put(89,-1){{\tiny\mbox{$>$}}}
\put(102,1){\circle*{3}}

\put(-2,-7){\mbox{\tiny $1$}}
\put(18,-7){\mbox{\tiny $2$}}
\put(38,-7){\mbox{\tiny $3$}}
\put(54,-7){\mbox{\tiny $\ell$$-$$2$}}
\put(75,-7){\mbox{\tiny $\ell$$-$$1$}}
\put(100,-7){\mbox{\tiny $\ell$}}
\end{picture}
\\

$C_\ell$ ($\ell \ge 2$)&none&
\begin{picture}(7,2)(0,0)
\put(0,1){\circle*{3}}
\put(0,1){\line(1,0){20}}
\put(20,1){\circle*{3}}
\put(20,1){\line(1,0){20}}
\put(40,1){\circle*{3}}
\put(43,-1.6){ \mbox{$\cdots$}}
\put(62,1){\circle*{3}}
\put(62,1){\line(1,0){20}}
\put(82,1){\circle*{3}}
\put(82,2){\line(1,0){20}}
\put(82,0){\line(1,0){20}}
\put(89,-1){{\tiny\mbox{$<$}}}
\put(102,1){\circle*{3}}

\put(-2,-7){\mbox{\tiny $1$}}
\put(18,-7){\mbox{\tiny $2$}}
\put(38,-7){\mbox{\tiny $3$}}
\put(54,-7){\mbox{\tiny $\ell$$-$$2$}}
\put(75,-7){\mbox{\tiny $\ell$$-$$1$}}
\put(100,-7){\mbox{\tiny $\ell$}}
\end{picture}

\\

$D_\ell$ ($\ell \ge 4$)&2&
\begin{picture}(7,2)(0,0)
\put(0,1){\circle*{3}}
\put(0,1){\line(1,0){20}}
\put(20,1){\circle*{3}}
\put(20,1){\line(1,0){20}}
\put(40,1){\circle*{3}}
\put(43,-1.6){ \mbox{$\cdots$}}
\put(62,1){\circle*{3}}
\put(62,1){\line(1,0){20}}
\put(82,1){\circle*{3}}
\put(82,2){\line(4,3){15}}
\put(82,0){\line(4,-3){15}}
\put(96.5,12.9){\circle*{3}}
\put(96.5,-10.9){\circle*{3}}

\put(-2,-7){\mbox{\tiny $1$}}
\put(18,-7){\mbox{\tiny $2$}}
\put(38,-7){\mbox{\tiny $3$}}
\put(54,-7){\mbox{\tiny $\ell$$-$$3$}}
\put(86,-0.5){\mbox{\tiny $\ell$$-$$2$}}
\put(100,-12){\mbox{\tiny $\ell$}}
\put(100,11.8){\mbox{\tiny $\ell$$-$$1$}}
\end{picture}

\\
$E_6$&2, 3&
\begin{picture}(7,2)(0,0)
\put(0,10){\circle*{3}}
\put(0,10){\line(1,0){15}}
\put(15,10){\circle*{3}}
\put(15,10){\line(1,0){15}}
\put(30,10){\circle*{3}}
\put(30,-5){\circle*{3}}
\put(30,10){\line(0,-1){15}}
\put(30,10){\line(1,0){15}}
\put(45,10){\circle*{3}}
\put(45,10){\line(1,0){15}}
\put(60,10){\circle*{3}}

\put(-2,3){\mbox{\tiny $1$}}
\put(13,3){\mbox{\tiny $3$}}
\put(33,3){\mbox{\tiny $4$}}
\put(43,3){\mbox{\tiny $5$}}
\put(58.5,3){\mbox{\tiny $6$}}
\put(33,-6){\mbox{\tiny $2$}}
\end{picture}

\\

$E_7$&2, 3&
\begin{picture}(7,2)(0,0)
\put(0,10){\circle*{3}}
\put(0,10){\line(1,0){15}}
\put(15,10){\circle*{3}}
\put(15,10){\line(1,0){15}}
\put(30,10){\circle*{3}}
\put(30,-5){\circle*{3}}
\put(30,10){\line(0,-1){15}}
\put(30,10){\line(1,0){15}}
\put(45,10){\circle*{3}}
\put(45,10){\line(1,0){15}}
\put(60,10){\circle*{3}}
\put(60,10){\line(1,0){15}}
\put(75,10){\circle*{3}}

\put(-2,3){\mbox{\tiny $1$}}
\put(13,3){\mbox{\tiny $3$}}
\put(33,3){\mbox{\tiny $4$}}
\put(43,3){\mbox{\tiny $5$}}
\put(58,3){\mbox{\tiny $6$}}
\put(33,-6){\mbox{\tiny $2$}}
\put(73,3){\mbox{\tiny $7$}}
\end{picture}

\\

$E_8$&2, 3, 5&
\begin{picture}(7,2)(0,0)
\put(0,10){\circle*{3}}
\put(0,10){\line(1,0){15}}
\put(15,10){\circle*{3}}
\put(15,10){\line(1,0){15}}
\put(30,10){\circle*{3}}
\put(30,-5){\circle*{3}}
\put(30,10){\line(0,-1){15}}
\put(30,10){\line(1,0){15}}
\put(45,10){\circle*{3}}
\put(45,10){\line(1,0){15}}
\put(60,10){\circle*{3}}
\put(60,10){\line(1,0){30}}
\put(75,10){\circle*{3}}
\put(90,10){\circle*{3}}

\put(-2,3){\mbox{\tiny $1$}}
\put(13,3){\mbox{\tiny $3$}}
\put(33,3){\mbox{\tiny $4$}}
\put(43,3){\mbox{\tiny $5$}}
\put(58,3){\mbox{\tiny $6$}}
\put(33,-6){\mbox{\tiny $2$}}
\put(73,3){\mbox{\tiny $7$}}
\put(88,3){\mbox{\tiny $8$}}

%
\end{picture}
\\

$F_4$&2, 3&
\begin{picture}(7,2)(0,0)
\put(0,1){\circle*{3}}
\put(0,1){\line(1,0){15}}
\put(15,1){\circle*{3}}
\put(15,0){\line(1,0){15}}
\put(15,2){\line(1,0){15}}
\put(19,-1){{\tiny\mbox{$>$}}}
\put(30,1){\circle*{3}}
\put(30,1){\line(1,0){15}}
\put(45,1){\circle*{3}}

\put(-2,-7){\mbox{\tiny $1$}}
\put(13,-7){\mbox{\tiny $2$}}
\put(28,-7){\mbox{\tiny $3$}}
\put(43,-7){\mbox{\tiny $4$}}
\end{picture} 
\\

$G_2$&2& 
\begin{picture}(7,2)(0,0)
\put(0,1){\circle*{3}}
\put(0,0.1){\line(1,0){15}}
\put(0,1.1){\line(1,0){15}}
\put(0,2.1){\line(1,0){15}}
\put(4,-1){{\tiny\mbox{$<$}}}
\put(15,1){\circle*{3}}

\put(-2,-7){\mbox{\tiny $1$}}
\put(13,-7){\mbox{\tiny $2$}}
\end{picture} \\[0.7ex] \bottomrule

\end{tabular}
\caption{Dynkin diagrams of simple root systems, with simple roots numbered as in \cite{Bou:g4}, and their torsion primes from \cite[p.~299]{Dem:inv}.} \label{dynks.table}
\end{center}
\end{table}

Irreducible representations $L(\la)$, $L(\la')$ of $G$ are equivalent \emph{up to graph automorphism} if there is an automorphism $\phi$ of the Dynkin diagram (i.e., an automorphism of the root system that normalizes the set of simple roots) so that $\phi(\la) = \la'$; such representations are equivalent up to an automorphism of $G$.  For example, the representations $\wedge^d k^n$ and $\wedge^{n-d} k^n$ of $\SL_n$ are equivalent up to graph automorphism.  Our results, which are about stabilizers in general position and so on, are transparently the same for representations that are equivalent up to graph automorphism.


\subsection*{Special characteristic}
For the remainder of this section, suppose that $G$ is simple and simply connected.  We say that $\car k$ is \emph{special} for $G$ if $G$ has type $G_2$ and $\car k = 3$ or $G$ has type $B_\ell$ ($\ell \ge 2$), $C_\ell$ ($\ell \ge 2$), or $F_4$ and $\car k = 2$.  (This was written as ``exceptionally bad characteristic'' in the title of \cite{GG:special}.)  That is, $\car k$ is special if the Dynkin diagram of $G$ has an edge with multiplicity $\car k$.  
When that holds, there is a \emph{very special isogeny} $\pi \!: G \to \Gb$ where $\Gb$ is also simple simply connected and the root system of $\Gb$ is inverse to the root system of $G$, see \cite[\S7.1]{CGP2} or \cite[\S10]{St:rep} for a concrete description.

Now suppose that $\car k$ is special for $G$, so in particular $\Delta$ has two root lengths.  Write a dominant weight $\la$ as  $\la = \sum c_\delta \omega_\delta$, where $c_\delta \ge 0$ and $\omega_\delta$ is the fundamental weight dual to $\delta^\vee$ for $\delta \in \Delta$.  We write $\la = \la_s + \la_\ell$ where $\la_s = \sum_{\text{$\delta$ short}} c_\delta \omega_\delta$ and $\la_\ell = \sum_{\text{$\delta$ long}} c_\delta \omega_\delta$, i.e., $\qform{\la_s, \delta^\vee} = 0$ for $\delta$ long and $\qform{\la_\ell, \delta^\vee} = 0$ for $\delta$ short.  Steinberg \cite{St:rep} shows that $L(\la) \cong L(\la_\ell) \ot L(\la_s)$ and that furthermore the action of $G$ on $L(\la_\ell)$ factors through the very special isogeny.  For example, in case $G = \Sp_{2\ell}$ for some $\ell \ge 2$ and $\car k = 2$, $\Gb = \Spin_{2\ell + 1}$ and the non-faithful representation $L(\omega_\ell)$ of $G$ is obtained by composing the very special isogeny $\pi$ with the spin representation of $\Gb$, which is irreducible and faithful. 

In particular, if $L(\la)$ is faithful, then $\la_s \ne 0$.

\section{Adjoint representation} 

We record in this section various results about $G$ acting on $\Lie(G)$ and the irreducible representation $L(\hst)$.

The generic stabilizer for simple $G$ acting on $\Lie(\Ad(G))$ is determined in \cite[Prop.~9.2]{GG:edp}; this representation has an s.g.p.\ whose identity component is a maximal torus in $G$.  It follows that $\dim k[\Lie(\Ad(G))]^G = \rank G$.  However, $\Lie(\Ad(G))$ agrees with $\g = \Lie(G)$ if and only if the center of $G$, i.e., the kernel of $G \to \Ad(G)$, is \'etale.

The general statement is that the Lie algebra $\gt$ of the simply connected cover $\Gt$ of $G$ is the Weyl module $V(\hst)$ with highest weight the highest root $\hst$ and the head of $V(\hst)$ is the irreducible representation $L(\hst)$, see for example \cite[2.5]{G:vanish}.

For $G$ simple, $\g$ is an irreducible representation of $G$ --- i.e., $\Lie(G) = L(\hst)$ --- if and only if the center $Z(\Gt)$  is \'etale and $\car k$ is not special for $G$, see \cite{Hiss}.

\begin{lem} \label{adj.eg}
The irreducible representation $L(\hst)$ of a simple group $G$ is faithful if and only if $G$ is adjoint, $\car k$ is not special for $G$, and $(G, \car k) \ne (\PGL_2, 2)$.  If those equivalent conditions hold, then the s.g.p.\ exists, its identity component is a maximal torus of $G$, and
\[
\dim k[L(\hst)]^G = \rank G - \dim \Lie(Z(\Gt)).
\]
\end{lem}

\begin{proof}
The kernel of the action of $\Gt$ on $V(\hst)$ is $Z(\Gt)$ \cite[Prop.~21.7]{Milne:AG}, so the kernel of the action of $G$ on $V(\hst)$ is $Z(\Gt) / \ker[\Gt \to G] \cong Z(G)$.  Thus,  for $L(\hst)$ to be faithful it is necessary that $G$ is adjoint. 

If $\car k$ is special for $G$, then $\n := \ker \dpi$ is a nonzero proper ideal in $\gt = V(\hst)$, so its image in $L(\hst)$ is zero.  In particular, the image of $\n$ in $\g$ acts trivially on $L(\hst)$.  As $\n \not\subseteq \Lie(Z(\Gt))$ \cite{Hiss}, its image in $\g$ is not zero, so $\g$ and the group scheme $G$ do not act faithfully on $L(\hst)$.

Suppose for the remainder of the proof that $G$ is adjoint and $\car k$ is not special for $G$.  Since $G(k)$ is simple as an ordinary group, it acts faithfully on $L(\hst)$, whence the kernel of the action of $G$ is infinitesimal.  

The socle of $\g$ as a representation of $G$ is irreducible \cite{Hiss}; it is the subalgebra generated by the root subspaces, i.e., the image of the Lie algebra of the simply connected cover of $\Gt$ of $G$, see for example \cite[Lemma 3.1(1)]{GG:large}.  That is, $L(\hst)$ is not faithful if and only if the composition $\gt \to \g  \to \gl(L(\hst))$ is zero if and only if 
$(G, \car k) = (\PGL_2, 2)$ by \cite[Lemma 3.1(2)]{GG:large}.
This concludes the proof of the first sentence of the lemma.  For the remainder of the proof we suppose additionally that $(G, \car k) \ne (\PGL_2, 2)$. 

Pick a maximal torus $\Tt$ in $\Gt$.  The orbit of a generic element in $V(\hst)$ meets $\Lie(\Tt)$ by \cite[XIII.5.1, XIV.3.18]{SGA3.2}, compare Lemma \ref{psi.sgp} below.
A generic element $t \in \Lie(\Tt)$ has image a generic element $v \in L(\hst)$.  The arguments of \cite[Example 3.4]{GG:large} apply and the stabilizer $G_v$ of $v$ has identity component the image $T$ of $\Tt$ and in particular $T \subseteq G_v \subseteq N_G(T)$.

An element $n \in N_G(T)(k)$ belongs to $G_v$ if and only if $\Ad(n)t - t$ is in the kernel of the map $V(\hst) \to L(\hst)$, i.e., is in $\Lie(Z(\Gt))$.  Equivalently, if and only if every root $\alpha$ vanishes on $\Ad(n)t - t$.  Since $t$ is generic, this is equivalent to $\alpha \circ \Ad(n) = \alpha$ for every root $\alpha$, equivalently, the element $\alpha \circ \Ad(n) - \alpha$ in the root lattice is divisible by $\car k$ for every simple root $\alpha$.  This last condition only depends on the image of $n$ in the Weyl group $N_G(T)/T$ and not on the choice of $t$, so $G_v$ depends only on the choice of $\Tt$ (equivalently, $T$), verifying that $G_v$ is an s.g.p.\ for the action of $G$ on $L(\hst)$.
\end{proof}

\begin{eg}
The s.g.p.\ appearing in the statement of Lemma \ref{adj.eg} need not be connected.  Take $G$ of type $E_8$, in which case the irreducible representation $L(\hst)$ is $\g$ itself and $G$ is both simply connected and adjoint.  In the notation of the proof, $G_v$ is connected if and only if $\car k \ne 2$, see \cite[Prop.~9.2]{GG:edp}.   (Note that Steinberg's result \cite[Th.~0.2]{St:tor}, which shows in some cases that semisimple elements have connected centralizers, assumes $\car k$ is not a torsion prime, so it does not apply here when $\car k$ is 2, 3, or 5.) 
  When $\car k  = 2$, $G_v / G^\circ_v \cong \Z/2$, where the nontrivial element acts on the torus $G_v^\circ$ by inversion.
 \end{eg}
 
 \begin{eg}
We apply Lemma \ref{adj.eg} to construct Table \ref{adj.small.table}, where we list the cases where simple $G$ acts faithfully on $L(\hst)$ and $\dim k[L(\hst)]^G \le 2$.  If $\dim \Lie(Z(\Gt)) = 0$, i.e., $Z(\Gt)$ is \'etale, then (type of $G, \car k)$ is one of $(A_1, {\ne 2})$, $(A_2, {\ne 3})$, $(B_2, {\ne 2})$, or $(G_2, {\ne 3})$.  If $\dim \Lie(Z(\Gt)) = 1$, then $\rank G \le 3$ and (type of $G, \car k)$ is $(A_2, 3)$ or $(A_3, 2)$.  Finally, if $\dim \Lie(Z(\Gt)) = 2$, then $G$ has type $D_{2m}$ for some $m \ge 2$, so (type of $G, \car k) = (D_4, 2)$.  
 \end{eg}
 
\begin{table}[ht]
\begin{tabular}{@{}crrc@{}} \toprule
type of $G$&$\car k$&$\dim L(\hst)$&$\dim k[L(\hst)]^G$ \\ \midrule
$A_1$&$\ne 2$&3&1 \\
$A_2$&$\ne 3$&8&2 \\
$A_2$&3&7&1\\
$A_3$&2&14&2\\
$B_2$&$\ne 2$&10&2 \\
$D_4$&2&26&2\\
$G_2$&$\ne 3$&14&2 \\ \bottomrule
\end{tabular}
\caption{Simple $G$ that act faithfully on $L(\hst)$ such that $\dim k[L(\hst)]^G \le 2$.} \label{adj.small.table}
\end{table}

We will use the following in the proof of Theorem \ref{same.thm} at the end of the paper, but we put it here because it only concerns $L(\hst)$ as a representation.

\begin{lem} \label{adj.sub}
Let $G$ be a simple algebraic group.  If there is a proper, connected, and  smooth algebraic subgroup of $G$ that acts irreducibly on $L(\hst)$, then $\car k$ is special for $G$.
\end{lem}

\begin{proof}
Let $G'$ be a connected, proper, and smooth subgroup of $G$ that acts irreducibly on $L(\hst)$.
Since the kernel of the action by $G$ on $L(\hst)$ is finite and $G'$ acts irreducibly, we conclude that $G'$ has trivial radical and so is semisimple.

For sake of contradiction, suppose $\car k$ is not special.
We may assume that $G$ is adjoint, so $L(\hst)$ is the socle of $\g$ as a representation of $G$ and $\g/L(\hst)$ is isomorphic to the Lie algebra of the center of the simply connected cover of $G$.
As the Lie algebra $\g'$ is a $G'$-invariant subspace of $\g$, it contains $L(\hst)$.  In particular, $G$ and $G'$ have the same number of roots and the same unipotent radicals for the Borel subgroups.  As these unipotent radicals of a semisimple group generate the group, we find $G = G'$.
\end{proof}

For the sake of completeness, we provide the following converse to Lemma \ref{adj.sub}.  In the statement, the hypotheses ``connected'' and ``smooth'' are redundant because they are included in the definition of simple (see \S\ref{back.sec}).  We have included them here to emphasize that this is a converse to the preceding.

\begin{lem}
Let $G$ be a simple algebraic group.  If $\car k$ is special for $G$, then there is a proper, connected, smooth, and simple subgroup of $G$ that acts irreducibly on $L(\hst)$.
\end{lem}

\begin{proof}
We may assume that $G$ is simply connected, so $L(\hst)$ is the head of $\g$.  

If $G$ is of type $C_n$ for $n \ge 2$, the root subgroups for the short roots generate a subgroup $G'$ of type $D_n$.  The restriction of $L(\hst)$ to $G'$ is a Frobenius twist of the natural representation of dimension $2n$, so $G'$ acts irreducibly on $L(\hst)$.

If $G$ has type $B_n$ for $n \ge 3$, $F_4$, or $G_2$, then the root subgroups for the long roots generate a subgroup $G'$ of type $D_n$, $D_4$, or $A_2$ respectively.  The description of $\g$ in \cite{Hiss} or \cite{Hogeweij} shows that $\Lie(G')$ maps onto $L(\hst)$ and indeed $L(\hst)$ is also the irreducible part of the adjoint representation of $G'$.
\end{proof}

\subsection*{Invariant polynomial functions}
We now study the rings of invariant polynomial functions $k[\g]^G$ and $k[L(\hst)]^G$, especially in the case where $\g \ne L(\hst)$.

Suppose $G$ is simple and let $T$ be a maximal torus in $G$.  For $W$ the Weyl group $N_G(T)/T$, the natural restriction map
\begin{equation} \label{CRT.map}
k[\g]^G \to k[\tor]^W
\end{equation}
 is injective.  Furthermore, it is an isomorphism if and only if  $(G, \car k) \ne (\Sp_{2\ell}, 2)$ for $\ell \ge 1$, see \cite[II.3.17']{SpSt}, \cite[p.~82, 7.12]{Jantzen:nil}, or \cite{ChaputRomagny}.
This is sometimes called the \emph{Chevalley Restriction Theorem}.  

If $G$ is simply connected, then $k[\tor]$ from the previous paragraph is the symmetric algebra on the weights with coefficients in $k$.  If $\car k$ is not a torsion prime for (the root system of) $G$ as in Table \ref{dynks.table}, then $k[\tor]^W$ is a polynomial algebra \cite[p.~297, Cor.]{Dem:inv}.  If additionally $(G, \car k) \ne (\Sp_{2\ell}, 2)$ for $\ell \ge 1$, then the generators of $k[\tor]^W$ have the same degrees as they do for the analogous group in case $k = \C$ \cite[p.~296, Cor., Th.~3]{Dem:inv}.

\begin{eg} \label{kg2}
Suppose $G$ has type $G_2$.  The only torsion prime for $G$ is 2, so when $\car k \ne 2$, $k[\g]^G$ is a polynomial ring with generators of degree 2 and 6 by the above and \cite[\S{VIII.8.3}]{Bou:g7}.

In case $\car k = 2$, we apply the Chevalley Restriction Theorem by hand.  The element $-1$ in the Weyl group acts trivially on $\tor$, so $k[\tor]^W$ is the ring of invariant functions of the symmetric group on 3 letters on its irreducible 2-dimensional representation.  We conclude that $k[\g]^G$ is a polynomial ring with generators of degrees 2 and 3.  (Alternatively, one can apply the next two examples to an $A_2$ subgroup containing $T$ to see that $k[\g]^G = k[\tor]^W = k[\tor]^{S_3} = k[\sl_3]^{\SL_3}$ to draw the same conclusion.)
\end{eg}

\begin{eg} \label{SL.ad}
For $G = \SL_n$, the ring of $G$-invariant functions on the space of $n$-by-$n$ matrices is polynomial with generators the coefficients of the characteristic polynomial.  
Tracking the proof of \cite[Prop.~4.1]{Nakajima}, one finds that the ring of invariants $k[\g]^G$ is also polynomial, with the same generators except for the trace.  Away from the case $n = 2$ and $\car k = 2$, the Chevalley Restriction Theorem and Demazure's result provides the same conclusion.

In case $n = 2$ and $\car k = 2$, the Weyl group $W$ acts trivially on the Lie algebra $\tor$ of a maximal torus in $G$, so $k[\tor]^W$ is a polynomial ring in one variable, generated by a linear function.  The image of the restriction map $k[\g]^G \to k[\tor]^W$ is $(k[\tor]^W)^{[2]}$.  \end{eg}

We now address $k[L(\hst)]^G$ in case $\g$ is a reducible representation.

\begin{eg} \label{coCRT}
Suppose $G$ is simple, $\car k$ is not special for $G$, and (type of $G, \car k) \ne (A_1, 2)$.  In particular, the Chevalley Restriction Theorem applies.  As the action of $G$ on $L(\hst)$ factors through the adjoint group $\Gb$ of $G$ and our goal is to calculate $k[L(\hst)]^G$, we are free to choose $G$ to be simply connected.  

Under the hypotheses, $L(\hst)$ is the image of $\g$ in $\gb$.   As $Z(G)$ is contained in $T$ \cite[Prop.~21.7]{Milne:AG}, its Lie algebra $\z$ is contained in $\tor$.  The natural maps $k[\tor/\z]^G \hookrightarrow k[\tor]^G$ and $k[\g/\z]^G \hookrightarrow k[\g]^G$ are compatible with the isomorphism \eqref{CRT.map} in the sense that it induces an isomorphism $k[L(\hst)]^G \xrightarrow{\sim} k[\tor_0]^W$, where $\tor_0 := \tor/\z$ is the image of $\tor$ in the Lie algebra of the image $\Tb$ of $T$.  It is the subspace of $\Lie(\Tb)$ spanned by the elements $h_\alpha$ in a Chevalley basis where $\alpha$ is a root.
(See  \cite[Example 8.3]{GG:simple} for a concrete illustration in the case $G$ is isogenous to $\SL_4$ and $\car k = 2$,  where $L(\hst)$ has a polynomial ring of invariants with generators of degrees 2 and 3.)
\end{eg}

\begin{cor} \label{adj.inv}
Suppose $G$ is simple and $L(\hst)$ is faithful.  If $\dim k[L(\hst)]^G \le 2$, then $k[L(\hst)]^G$ is a polynomial ring.
\end{cor}

\begin{proof}  
The type of $G$ and $\car k$ appear as a row in Table \ref{adj.small.table}.
If $\dim k[L(\hst)]^G = 1$, then $k[L(\hst)]^G$ is a polynomial ring for dimension reasons, \cite[Prop.~6.1]{BGL}, so assume $\dim k[L(\hst)]^G = 2$.

If $Z(\Gt)$ is \'etale, we apply the Chevalley Restriction Theorem.  Otherwise (type of $G, \car k) = (A_3, 2)$ or $(D_4, 2)$ and we apply Example \ref{coCRT}.  In either case, $k[L(\hst)]^G$ is isomorphic to the ring of invariant polynomials of a pseudoreflection group (the Weyl group) acting on a 2-dimensional space, so it is a polynomial ring by \cite[Th.~5.1]{Nakajima} or \cite[Prop.~7.1]{KemperMalle}.
\end{proof}

See Proposition \ref{adj.cofree} below for a stronger version of Corollary \ref{adj.inv}.

\begin{prop} \label{pgl.ad}
Let $G = \PGL_n$ for some $n \ge 2$ over a field $k$.  The following are equivalent:
\begin{enumerate}
\item \label{pgl.ad.hst} $k[L(\hst)]^G$ is a polynomial ring.
\item \label{pgl.ad.ad} $k[\g]^G$ is a polynomial ring.
\item \label{pgl.ad.n} $n \le 4$ or $\car k$ does not divide $n$.
\end{enumerate}
\end{prop}

\begin{proof}
Suppose first that $\car k$ does not divide $n$.  Then the representations $\sl_n$, $\pgl_n$, and $L(\hst)$ are all naturally identified with each other and $k[\sl_n]^{\SL_n}$ is a polynomial ring as in Example \ref{SL.ad}.

Therefore, we restrict our attention to the case where $\car k$ does divide $n$, applying Example \ref{coCRT} for $k[L(\hst)]^G$ and the Chevalley Restriction Theorem for $k[\g]^G$.  As the roots of $G$ all have the same length, we may identify the root system with its dual and so identify the toral subalgebra $\tor$ of $\g$ with the weight lattice tensored with $k$ and $\tor_0$ with the subspace spanned by the roots. 

If we identify the Weyl group with the symmetric group $S_n$ and consider its permutation representation $X$ with basis $x_1, \ldots, x_n$, then the ring $k[\tor]$ of functions on $\tor$ with the action by the Weyl group is identified with the symmetric algebra $A_Y$ on the subspace $Y$ of $X$ of elements whose coordinates sum to zero.  Moreover, the space of characters on $T$ that vanish on $\tor_0$ are a 1-dimensional subspace of $Y$ and contain $\sum x_i$, so the functions on $k[\tor_0]$ are identified with the symmetric algebra $A_Z$ on $Z := Y/ k\sum x_i$.  If $n > 4$, then $A_Y^{S_n}$ and $A_Z^{S_n}$ are not polynomial rings, see \cite[\S4]{Nakajima} and \cite[\S5]{KemperMalle}, proving the claim in that case.

If $n \le 4$, then $\dim k[L(\hst)]^G \le 2$ and
we have already observed that $k[L(\hst)]^G$ is a polynomial ring in Corollary \ref{adj.inv}.  If $n = 2$ or 3, then $k[\g]^G$ is by the Chevalley Restriction Theorem the functions in 1 or 2 variables that are invariant under a finite reflection group, so for the same reason we conclude that $k[\g]^G$ is a polynomial ring. 

Finally, consider $k[\g]^G$ in the case $n = 4$ and $k = \F_2$.  The ring $A_Y$ has generators
\begin{equation} \label{AXY}
y_1 = x_1 - x_2, \quad y_2 = x_2 - x_3, \quad y_3 = x_3 - x_4.
\end{equation}
Set
\begin{gather*}
f_1 = y_1 + y_3,\\
f_2 = y_1^2 y_2^2 + y_2^4 + y_1^2 y_2 y_3 + y_1 y_2^2 y_3 + y_1^2 y_3^2 + y_1 y_2 y_3^2 + 
 y_2^2 y_3^2, \quad \text{and} \\
  f_3 = y_1^4 y_2^2 + y_1^2 y_2^4 + y_1^4 y_2 y_3 + y_1 y_2^4 y_3 + y_1^4 y_3^2 + 
 y_1^2 y_2^2 y_3^2 + y_2^4 y_3^2 + y_1^2 y_3^4 + y_1 y_2 y_3^4 + y_2^2 y_3^4
 \end{gather*}
in $A_Y$.  Rewriting these using \eqref{AXY}, we find that $f_1 = \sum x_i$ is in $A_Y^{S_4}$ and similarly for $f_2$ and $f_3$.  The determinant of the Jacobian matrix with $(i, j)$ entry $\partial f_i / \partial y_j$ is not zero (e.g., the term $y_1^5 y_2^2 y_3$ appears) and $\prod \deg f_i = | S_4 |$, so the $f_i$ are algebraically independent and $k[\g]^G \cong A_Y^{S_4} = k[f_1, f_2, f_3]$ is a polynomial ring by the criterion from \cite[Th.~3.9.4]{DerksenKemper}.
\end{proof}

\begin{eg} \label{not.cofree}
Let $V$ be an irreducible and faithful representation of a simple algebraic group $G$ over $k$.  One can ask whether the property that $G_v \ne 1$ for generic $v \in V$ (denoted by (ST) in \cite[\S8]{PoV}) is equivalent to the property that $k[V]^G$ is a polynomial ring (denoted by (FA) in ibid.).  If $\car k = 0$, then the two properties are equivalent, see \cite[Th.~8.8]{PoV} and \cite{KPV}.

If $G = \PGL_n$ such that $\car k$ divides $n$ and $n \ge 5$ and $V = L(\hst)$, then $(G, V)$ satisfies (ST) (Lemma \ref{adj.eg}) but not (FA) (Proposition \ref{pgl.ad}).

Alternatively, take $G = \PSp_{2n}$ with the same hypotheses on $n$.  The representation $V = L(\omega_2)$ is faithful and irreducible and $G_v \ne 1$ for generic $v \in V$, see Table \ref{meta.small}.  Example 8.5 in \cite{GG:simple} shows that $k[V]^G$ is isomorphic to the ring of invariants in the preceding paragraph, so it too is not a polynomial ring.
\end{eg}

\section{Summary of some recent results} \label{summ.sec}

Recent results of the authors, which rely on L\"ubeck's paper \cite{luebeck}, combine with that paper to give the following results.
\begin{thm} \label{summ.thm}
Let $V$ be an irreducible and faithful representation of a simple linear algebraic group $G$ over an algebraically closed field $k$.
\begin{enumerate}
\item \label{summ.big} Suppose $\dim V > \dim G$.  Then $G_v$ is finite for generic $v \in V$, and $\dim k[V]^G = \dim V - \dim G$.  Moreover, $G_v \ne 1$ for generic $v \in V$  if and only if $(G, \car k, V)$ is listed in Table \ref{meta.big} up to graph automorphism. 
\item \label{summ.eq} If $\dim V = \dim G$, then $G$ is adjoint, $V$ is the adjoint representation $\Lie(G) = L(\hst)$,
$\car k$ is not special for $G$,
there is a s.g.p.\ (whose identity component is a maximal torus),
and $\dim k[V]^G = \rank G$.
\item \label{summ.small} If $\dim V < \dim G$, then $(G, \car k, V)$ is listed in Table \ref{meta.small} up to graph automorphism or $V = L(\hst)$.
\end{enumerate}
In particular, $G_v$ is finite for generic $v \in V$ if and only if $\dim V > \dim G$.
\end{thm}

We now justify the columns in Tables \ref{meta.small} and \ref{meta.big} concerning $G_v$.
The $k$-points $G(k)_v = G_v(k)$ for generic $v \in V$ are from \cite{GurLawther}.

 In Table \ref{meta.big}, the Lie algebra $\g_v$ was computed in  \cite{GG:irred} and \cite{GG:special}.  Since $G(k)_v$ is finite in all cases, if additionally $\g_v = 0$ then $G_v$ is \'etale (in particular, smooth) and so completely described by $G(k)_v$.  In particular, $G_v$ is commutative if and only if $G(k)_v$ is an abelian group.  If $\g_v \ne 0$, then $G_v$ is not smooth, and we study $G_v$ in \S\ref{vinberg.sec} and \S\ref{inf.sec}.


\begin{table}[htb]
\begin{tabular}{@{}ccccc@{}}\toprule
$G$&$\car k$&$V$&$G_v(k)$&$\dim k[V]^G$ \\ \midrule\midrule
$A_\ell$ ($\ell \ge 1$)&any&$k^{\ell+1}$&$A_{\ell-1} U_\ell$&0 \\
$A_\ell$ ($\ell \ge 1$)&$\ne 2$&$S^2(k^{\ell+1})$&$D_{(\ell+1)/2}$ or $B_{\ell/2}$&1 \\
$A_\ell$ (odd $\ell \ge 3$)&any&$\wedge^2(k^{\ell+1})$&$C_{(\ell+1)/2}$&1 \\
$A_\ell$ (even $\ell \ge 4$)&\multicolumn{2}{c}{$\cdots$ same as line above $\cdots$}&$C_{\ell/2} U_\ell$&0 \\
$\SL_6/\mu_3$&any&$\wedge^3 k^6$&$A_2^2 . \Z/{(2,p)}$& 1 \\
$\SL_7$&any&$\wedge^3 k^7$&$G_2$&1 \\
$\SL_8$&any&$\wedge^3 k^8$&$A_2.\Z/{(2,p)}$&1 \\ \midrule
$\SO_{2\ell+1}$ ($\ell \ge 2$)&$\ne 2$&$L(\omega_1)$ (natural)&$D_\ell$ &1\\
$\Spin_7$&any&$L(\omega_3)$ (spin)&$G_2$&1 \\
$\Spin_9$&any&$L(\omega_4)$ (spin)&$B_3$&1 \\
$\Spin_{11}$&any&$L(\omega_5)$ (spin)&$A_4.\Z/{(2,p)}$&1 \\
$\Spin_{13}$&any&$L(\omega_6)$ (spin)&$A_2^2.(\Z/{(2,p)})^2$&2 \\ \midrule
$\Sp_{2\ell}$ ($\ell \ge 2$)&any&$L(\omega_1)$ (natural)&$C_{\ell-1} U_{2\ell-1}$&0 \\
$\Sp_6$&$\ne 2$&$L(\omega_3)$ (``spin'')&$\tilde{A}_2$&1 \\   
$\PSp_6$&any&$L(\omega_2)$&$C_1^3 . \Z/(3,p)$&$2-\eps$ \\
$\PSp_8$&\multicolumn{2}{c}{$\cdots$ same as line above $\cdots$}&$C_1^4.(\Z/(2,p))^2$& $3-\eps$ \\ 
$\PSp_{2\ell}$ ($\ell \ge 5$)&\multicolumn{2}{c}{$\cdots$ same as line above $\cdots$}&$C_1^\ell$& $\ell-1-\eps$ \\ \midrule
$\SO_{2\ell}$ ($\ell \ge 4$)&any&$L(\omega_1)$ (natural)&$B_{\ell-1}$& 1 \\
$\Spin_{10}$&any&spin&$B_3U_8$&0 \\
$\HSpin_{12}$&any&half-spin&$A_5.\Z/{(2,p)}$&1 \\
$\Spin_{14}$&any&spin&$G_2^2.\Z/{(2,p)}$&1 \\ \midrule
$G_2$&$\ne 2$&$L(\omega_1)$ (natural)&$A_2$&$1$ \\ 
$G_2$&$2$&same as line above&$A_1U_5$&$0$ \\
%
$F_4$&any&$L(\omega_4)$ (natural)&$D_4$&$2-\eps$ \\  
%
$E_6$&any&$L(\omega_1)$ (minuscule)&$F_4$&1 \\
%
$E_7$&any&$L(\omega_7)$ (minuscule)&$E_6.\Z/{(2,p)}$&1 \\ \bottomrule
%
\end{tabular}
\caption{Irreducible faithful representations $V$ for a simple algebraic group $G$ over an algebraically closed field $k$ of characteristic $p$ such that $\dim V \le \dim G$, except for $L(\hst)$ and up to graph automorphism.  The symbol $\eps$ represents 0 or 1, where the value is determined by $\ell$ and $\car k$.} \label{meta.small}
\end{table}


\begin{sidewaystable}
\begin{centering}
\begin{tabular}{@{}cccccccc@{}} \toprule
$G$&$\car k$&$V$&$\dim V$&$G_v$&$G_v$ smooth?&$G_v$ commutative?&$\dim k[V]^G$ \\ \midrule\midrule
$\SL_2$&$\ne 2, 3$&$S^3 k^2$&4&$\Z/3$ &yes& yes&1\\
$\SL_2/\mu_2$&$\ne 2, 3$&$S^4 k^2$&5&$(\Z/2)^2$ & yes&yes&2\\
$\SL_3/\mu_3$&$\ne 2, 3$&$S^3 k^3$&10 &$(\Z/3)^2$ & yes&yes&2\\
$\SL_4$&3&$L(\omega_1 + \omega_2)$&16&$\Alt_5$&yes&no&1 \\
$\SL_4/\mu_4$&$\ne 2$&$L(2\omega_2)$&$20-\eps$&$(\Z/2)^4$ & yes& yes&$5-\eps$ \\
$\SL_4$&$p$ odd&$L(p^e \omega_1 + \omega_2)$, $e \ge 1$&24&$\mu_{p^e}$&no&yes&9 \\
$\SL_4/\mu_2$&2&$L(2^e \omega_1 + \omega_2)$, $e \ge 2$&24&$\mu_{2^e}$&no&yes&9 \\
$\SL_8/\mu_4$&$\ne 2$&$\wedge^4 k^8$&70&$(\Z/2)^6$&yes&yes&7 \\
$\SL_8/\mu_4$&$2$&\multicolumn{2}{c}{$\cdots$ same as line above $\cdots$}&$(\Z/2)^3 \times (\mu_2)^3$&no&yes&7 \\
$\SL_9/\mu_3$&$\ne 2, 3$&$\wedge^3 k^9$&84&$(\Z/3)^4$&yes&yes&4 \\
$\SL_9/\mu_3$&2&\multicolumn{2}{c}{$\cdots$ same as line above $\cdots$}&$(\Z/3)^4.(\Z/2)$&yes&no&4\\
$\SL_9/\mu_3$&3&\multicolumn{2}{c}{$\cdots$ same as line above $\cdots$}&$(\Z/3)^2 \times (\mu_3)^2$&no&yes&4\\
$A_\ell$&$p \ne 0$&$L(\omega_1 + p^e \omega_\ell)$, $e \ge 1$&$(\ell +1)^2$&$\PSU_{\ell+1}(p^e)$ & yes &no&1 \\
$A_\ell$&$p \ne 0$&$L(\omega_1 + p^e \omega_1)$, $e \ge 1$&$(\ell +1)^2$&$\PSL_{\ell+1}(p^e)$ & yes&no&1 \\ \midrule
$\Spin_5$&5&$L(\omega_1 + \omega_2)$&12&$\mu_5$&no&yes&2 \\
$\SO_{2\ell+1}$ ($\ell \ge 2$) & $\ne 2$&$L(2\omega_1)$ (``$S^2$'') &$2\ell^2+3\ell-\eps$ & $(\Z/2)^{2\ell}$& yes&yes&$2\ell-\eps$ \\ \midrule
$\Sp_8$&3&$L(\omega_3)$ (``$\wedge^3$'')&40&$(\mu_3)^2$&no&yes&4 \\
$\PSp_8$&$\ne 2$&$L(\omega_4)$ (``spin'')&$42-\eps$&$(\Z/2)^6$ & yes&yes&$6-\eps$ \\ \midrule
%
$\SO_{2\ell}/\mu_2$ ($\ell \ge 4$)&$\ne 2$&$L(2\omega_1)$ (``$S^2$'')&$2\ell^2+\ell-1-\eps$&$(\Z/2)^{2\ell - 2}$ &yes&yes&$2\ell-1-\eps$ \\
$\HSpin_{16}$&$\ne 2$&half-spin&128&$(\Z/2)^8$&yes&yes&8 \\
$\HSpin_{16}$&2&\multicolumn{2}{c}{$\cdots$ same as line above $\cdots$}&$(\Z/2)^4 \times (\mu_2)^4$&no&yes&8 \\ \bottomrule
\end{tabular}
\caption{Irreducible faithful representations $V$ for a simple algebraic group $G$ over an algebraically closed field $k$ such that $\dim V > \dim G$ and $G_v \ne 1$, up to graph automorphism.  In the dimensions, $\eps$ represents 0 or 1; the value it takes depends on $\ell$ and $\car k$.} \label{meta.big}
\end{centering}
\end{sidewaystable}


\begin{sidewaystable}
\begin{tabular}{@{}ccccccc@{}} \toprule
$G$&$\car k$&$V$&$\dim V$&$\dim k[V]^G$ \\ \midrule\midrule
$\SL_3$&$\ne 2, 3$&$S^4k^3$&15&7 \\
$\SL_3$&$\ne 2$&$L(\omega_1+2\omega_2)$&15&7 \\
$A_3$&$p \ne 0$&$L(p^e\omega_1 + \omega_2)$ or $L(\omega_1 + p^e\omega_2)$ for $e \ge 1$&24&9 \\
$\SL_5$&3&$L(2\omega_2)$&45&21 \\
$\SL_4$&$\ne 3$&$L(\omega_1 + \omega_2)$&20&5 \\
$\SL_5$&$\ne 3$&same as line above&40&16 \\
$\SL_n/\mu_{(n,3)}$ ($5 \le n \le 9$) & 3 & same as line above & $(n-1)n(n+4)/6$ & $(n^3  - 3 n^2 - 4 n + 6)/6$ \\
$\SL_5$&$\ne 2$&$L(\omega_1 + \omega_3)$&45&21 \\
$\SL_5$&2&same as line above&40 & 16 \\
$\SL_9$ & all & $\wedge^4 k^9$&126&46 \\
$\SL_n/\mu_{(n,3)}$ ($4 \le n \le 8$) & $\ne 2, 3$& $S^3 k^n$ & $\binom{n+2}{3}$ & $\binom{n+2}{3} - n^2 + 1$ \\
$\SL_n/\mu_{(n,3)}$ ($10 \le n \le 14$) & all & $\wedge^3 k^n$ & $\binom{n}{3}$ & $\binom{n}{3} - n^2  + 1$ \\ \midrule  
$\Spin_{5}$&$\ne 5$&$L(\omega_1 + \omega_2)$&16&6 \\
$\Spin_{7}$&$\ne 2$&$L(2\omega_3)$&35&14\\
$\Spin_7$&7&$L(\omega_1 + \omega_3)$&40&19 \\
$\Spin_{15}$&all&$L(\omega_7)$ (spin)&128&23\\
$\Spin_{17}$&all&$L(\omega_8)$ (spin)&256&120\\ \midrule 
$\Sp_{4}$&$p \ne 0$&$L(\omega_1 + p^e \omega_1)$ for $e \ge 1$&16&6\\
$\Sp_8$&$\ne 3$&$L(\omega_3)$ (``$\wedge^3$'')&48&12 \\
$\Sp_{10}$&2&same as line above&100&45 \\ \midrule
$D_4$&2&$L(\omega_1 + \omega_2)$&48&20 \\ \midrule
$G_2$&$\ne 2$&$L(2\omega_2)$&$27-\eps$ & $13-\eps$ \\ \bottomrule
\end{tabular}
\caption{Some irreducible faithful representations $V$ for a simple algebraic group $G$ over an algebraically closed field $k$ such that $\dim G < \dim V < 2 \dim G$.  Combining with Table \ref{meta.big} provides a full list up to graph automorphism.  The minimum value for $\dim k[V]^G$ in this table is 5, occurring for $(\SL_4, {\ne 3}, L(\omega_1 + \omega_2))$ and for $(\SL_4, {\ne 2, 3}, S^3 k^4)$.} \label{meta.few}
\end{sidewaystable}

\begin{proof}
\eqref{summ.big}: Assume $\dim V > \dim G$.  The rows in Table \ref{meta.big} are a union of the rows in Table 1 in each of \cite{GG:special} (those with $\g_v \ne 0$) and \cite{GurLawther} (those with $G_v(k) \ne 1$), although we have omitted those entries corresponding to non-faithful representations, such as spin representations of $\Sp_{2\ell}$ when $\car k = 2$.  Conversely, if $\g_v = 0$ and $G(k)_v = 1$, then $G_v = 1$.  

Note that in each row of Table \ref{meta.big}, $G_v(k)$ is finite.  This gives the claim on $\dim k[V]^G$.

For the remainder of the proof, we assume that $\dim V \le \dim G$, so the highest weight of $V$ is restricted by an easy dimension argument as in \cite[Lemma 1.1]{GG:irred}.  That is, $V$ is among the representations enumerated in \cite{luebeck}.

If $\dim V = \dim G$, checking the tables in \cite{luebeck} verifies that $V$ is the adjoint representation and $\car k$ is not special.  Moreover, as $\dim V = \dim G$, the center of the simply connected cover of $G$ is \'etale and the generic stabilizer is computed in \cite[Prop.~9.2]{GG:edp}.  This verifies \eqref{summ.eq}.

For \eqref{summ.small}, the list in Table \ref{meta.small} is somewhat shorter than in \cite{luebeck}, because we have omitted those representations that factor through the very special isogeny, i.e., those $\la$ that vanish on the short simple roots such as the spin representations of type $C$ when $\car k = 2$.

For the final claim, note that for $V = L(\hst)$, $G_v$ is not finite for generic $v \in V$ by Lemma \ref{adj.eg}.
\end{proof}

\begin{rmk} \label{Spin13}
In Table \ref{meta.big}, the stabilizer $G_v$ in $\Spin_{13}$ is $(\SL_3 \times \SL_3) \rtimes (\Z/2 \times \Z/2)$, as described in \cite[Prop.~5.2.9]{GurLawther}.  The $\Z/2$'s are generated by an element that acts as an outer automorphism on each $\SL_3$ and an element that interchanges the two $\SL_3$'s.  This corrects a mistake in \cite[Prop.~9.2]{GG:spin}, where one of the $\Z/2$ factors was omitted.
\end{rmk}

The following is an analogue for group schemes of a result proved for $G(k)$ in \cite[Cor.~11]{GurLawther}.  Recall that an algebraic group $G$ is said to act \emph{generically freely} on $V$ if $G_v = 1$ for generic $v \in V$.

\begin{cor} \label{thm.ross}
Let $G$ be a simple linear algebraic group acting faithfully and irreducibly on a representation $V$.  If $V$ has a nonzero weight space with multiplicity $> 1$, then $G$ acts generically freely on $V$.
\end{cor}

\begin{proof}
We apply Theorem \ref{summ.thm} and verify that the nonzero eigenvalues have multiplicity 1 in each faithful irreducible representation $V = L(\la)$ that is not generically free.  

The list of $L(\la)$ where the nonzero eigenvalues have multiplicity 1 for $k$ of prime characteristic has been 
obtained in \cite[Prop.~8]{TestermanZalesski}, leveraging \cite[6.1]{seitzmem} and \cite{ZalesskiSuprunenko}.

Alternatively, for each $\la$, if the irreducible representation over $\C$ with the same highest weight has nonzero eigenvalues of multiplicity 1 (the list of such is known from \cite[Th.~4.6.3]{Howe:wmf}), then we are done.  The remaining $\la$ can be treated in an ad hoc manner.

Yet another way to phrase the proof is to suppose that $V$ has a nonzero weight space with multiplicity 1 and apply 
\cite[Cor.~11]{GurLawther} to deduce that $G_v(k) = 1$ for generic $v \in V$.  This reduces the proof to verifying that nonzero eigenvalues have multiplicity 1 in the faithful irreducible representations that have nonzero but infinitesimal generic stabilizer.  There are just four of these, appearing in Table \ref{meta.big}.  (While the generic stabilizers for these cases are determined later in this paper, at this point we only need to know which representations have infinitesimal generic stabilizers, which is known from combining the results of \cite{GurLawther} and \cite{GG:special}.)
\end{proof}

\begin{cor}
Let $V$ be an irreducible representation of a simple algebraic group $G$.  Then either $G$ has an open orbit in $V$ or there is a dense open subset of $V$ consisting of closed $G$-orbits.
\end{cor}

\begin{proof}
As the $G$-orbits in $V$ are not changed by replacing $G$ by a quotient, we may assume that $G$ acts faithfully.  Suppose $G$ does not have an open orbit in $V$, i.e., $\dim k[V]^G > 0$.  Then by the theorem and Tables \ref{meta.small} and \ref{meta.big},  for generic $v \in V$, the group $G_v(k)$ is the $k$-points of a reductive group, i.e., $(G_v)_{\red}$ is reductive.  Thus the quotient $G/(G_v)_{\red}$ is affine \cite{Richardson:Matsushima}, whence the claim by \cite{Popov:stability}. 
\end{proof}

\section{Stabilizers in general position} \label{sgp.sec}

For a vector $v$ in a representation $V$ of an algebraic group $G$, one can consider separately (1) $G_v$, the stabilizer of $v$ in $G$ as a closed sub-group-scheme of $G$; (2) $G(k)_v$, the stabilizer of $v$ in the abstract group $G(k)$ of $k$-points of $G$; or (3) $\g_v$, the annihilator of $v$ in the Lie algebra $\g$ of $G$.  The first carries at least as much information as the other two, in the sense that
\[
G_v(k) = G(k)_v \quad \text{and} \quad \Lie(G_v) = \g_v.
\]
So far, we have focused on the notion of stabilizer in general position in the sense of group schemes, i.e., (1).  One can also consider versions for (2) and (3), namely:

\begin{defn}
For the group of $k$-points, $G(k)$, we say that a stabilizer in general position (s.g.p.) exists if there is a subgroup $G(k)_*$ of $G(k)$ and a dense open subset $U$ of $V$ so that for every $u \in U(k)$, there is a $g \in G(k)$ such that $gG(k)_ug^{-1} = G(k)_*$.

An s.g.p.\ for the Lie algebra $\g$ of $G$ is a subalgebra $\g_*$ of $\g$ such that for every $u \in U(k)$, there is a $g \in G(k)$ such that $(\Ad g) \g_u = \g_*$.
\end{defn}

Recall that, if $\car k = 0$ and $G$ is reductive, then a stabilizer in general position $\sgp$ exists by \cite{Richardson:prin}, see \cite[\S7]{PoV} for a survey.   Indeed, an s.g.p.\ exists even for the action of reductive $G$ on a smooth affine variety $X$ when $\car k = 0$.  However, the same claim does not hold in prime characteristic, even on the level of $k$-points, as the following example demonstrates.

\begin{eg} \label{no.sgp}
The generic stabilizer need not exist when $G$ is semisimple.  For example, take $G = \Spin_7 \times \SL_4$, $\car k = 2$, and $V = (\text{spin}) \ot k^4$.  Combining Prop.~6.2.9 and the proof of Lemma 4.6.1 in \cite{GurLawther} shows that for generic $y \in \P(V)$, the stabilizer $G_y(k)$ is semisimple with $k$-points of type $B_1 \times B_1$, but that there is no s.g.p. for the action of $G(k)$ on $\P(V)$.  (That is, there is a dense open subset of $\P(V)$ on which the stabilizers in $G(k)$ are all isomorphic, but there is not one where they are conjugate under $G(k)$.)  As $G_y(k)$ is semisimple, it follows that the stabilizer $G_v$ of a generic $v \in V$ has the same $k$-points.  Therefore, if the action of $G(k)$ on $V$ has an s.g.p., then so does the action of $G(k)$ on $\P(V)$ (and it would be the same s.g.p.), a contradiction.
\end{eg}

The following lemma allows us to leverage the results of Guralnick-Lawther (where the s.g.p.\ was computed for the abstract group $G(k)$) and Garibaldi-Guralnick (where the generic stabilizer was computed for the Lie algebra $\g$).

\begin{lem} \label{smooth.sgp}
Let $G$ be a group scheme over an algebraically closed field $k$ acting on an irreducible variety $X$ so that 
\begin{enumerate}
\item there is an s.g.p.\ $G(k)_*$ for the action by the abstract group of points $G(k)$ and
\item \label{smooth.sgp.2} there is an $x \in X(k)$ such that $\dim \g_x = \dim G(k)_*$.
\end{enumerate}
Then there is an s.g.p.\ for the action of the group scheme $G$ on $X$ and it is smooth.
\end{lem}

\begin{proof}
The set $\{ u \in X \mid \dim \g_u \le \dim G(k)_* \}$ is open by upper semicontinuity of dimension and it is nonempty by \eqref{smooth.sgp.2}.  Put $U$ for its intersection with a nonempty open subset of $X$ consisting of $u$ such that $G(k)_u$ is conjugate to $G(k)_*$.  For any $u \in U(k)$, we have
\[
\dim G(k)_u \le \dim \g_u \le \dim G(k)_*
\]
where the first inequality holds because $G_u$ is an algebraic group and the second is by construction of $U$.  As $G(k)_u$ and $G(k)_*$ are conjugate, they have the same dimension, and we have verified that $G_u$ is smooth.

For $u, u' \in U(k)$, there is a $g \in G(k)$ so that $g G_u g^{-1}$ and $G_{u'}$ have the same $k$-points, by construction of $U$.  As both group schemes are smooth, they agree \cite[Prop.~3.16]{Milne:AG}.
\end{proof}

\begin{eg}[Type $A_1$] \label{A1}
Suppose $G$ has type $A_1$ and $V$ is a faithful irreducible representation with highest weight $\la$, a natural number.  The hypothesis that $G$ is faithful says that $\car k$ does not divide $\la$ (for otherwise $V$ is a Frobenius twist of another representation) and that $G = \SL_2$ if and only if $\la$ is odd.

If $\la = 1$, then $V$ is the tautological representation of $\SL_2$, which has an open orbit ($\dim k[V]^G = 0$), and therefore there is an s.g.p.

If $\la = 2$ (so $\car k \ne 2$), then $V$ is the adjoint representation and the stabilizer of a generic element is a maximal torus.  We find $\dim k[V]^G = 1$.

If $\la \ge 3$, then $\g_v = 0$ for generic $v$ by \cite[Examples 1.8 and 3.3]{GG:large}.  As in \cite[Th.~2]{GurLawther}, $G_v(k)$ is finite, and $\ne 1$ only for the cases in Table \ref{meta.big}.  By Lemma \ref{smooth.sgp}, there is an s.g.p.\ for the action of $G$ on $V$.  In this case, $\dim k[V]^G = \dim V - 3$.

In summary, there is an s.g.p.\ for the action of $G$ on $V$.
\end{eg}

Here is another application of Lemma \ref{smooth.sgp}.

\begin{lem} \label{sp.wedge2}
For $G = \PSp_{2\ell}$ with $\ell \ge 3$ and $V$ the Weyl module $V(\omega_2)$ or the irreducible module $L(\omega_2)$, an s.g.p.\ exists and it is smooth.
\end{lem}

\begin{proof}  
Proposition 5.2.5 of \cite{GurLawther} shows that an s.g.p.\ exists for the action of $G(k)$ on $V(\omega_2)$ and on $L(\omega_2)$, and it has dimension $3\ell$.  Therefore, in view of Lemma \ref{smooth.sgp}, it suffices to verify that $\dim \g_v = 3\ell$ for $v$ generic in $V(\omega_2)$ or $L(\omega_2)$.

View the simply connected cover $\Sp_{2\ell}$ of $G$ as the subgroup of $\GL_{2\ell}$ preserving the alternating bilinear form $s(m,m') := m^\top J m'$ where $J = \stbtmat{0}{I_\ell}{-I_{\ell}}{0}$.  The group $\GSp_{2\ell}$ of similarities of $s$ is generated by $\Sp_{2\ell}$ and scalar transformations.  Its Lie algebra consists of those $x \in \gl_{2\ell}$ of the form $\stbtmat{A}{B}{C}{\mu I_\ell - A^\top}$ for $A, B, C \in \gl_\ell$ and $\mu \in k$ such that $B^\top = B$ and $C^\top = C$ as in \cite[Example 8.1]{GG:special}.  There is a natural exact sequence $1 \to \Gm \to \GSp_{2\ell} \to G \to 1$ and the corresponding map $\gsp_{2\ell} \to \g$ is surjective.  Therefore, it suffices to prove that $\dim (\gsp_{2\ell})_v = 3\ell + 1$ for generic $v \in V$.

We first treat the case $V = V(\omega_2)$.
Let $Y$ be the space of self-adjoint operators with respect to $s$, i.e., those transformations $y$ of $k^{2\ell}$ such that $s(ym,m) = 0$ for all $m \in k^{2\ell}$.  It is a representation of $G$ under the action $\rho(g) y = gyg^{-1}$.  The ring of $G$-invariant polynomial functions $k[Y]^G$ has a linear generator that sends $y \in Y$ to half the trace of $Jy$, see \cite[Example 8.5]{GG:simple} and the kernel of this linear map is $V(\omega_2)$.
Take $V_1$ to be the subspace of $V(\omega_2)$ consisting of diagonal matrices.  The map $G \times V_1 \to V$ is dominant \cite[Cor.~2.10]{GoGu}, so we may take for generic $v \in V$ a diagonal matrix $v = \stbtmat{L}{0}{0}{L}$ where $L \in \gl_\ell$ is diagonal with entries $(\lambda_1, \ldots, \lambda_\ell)$ such that $\sum_{i=1}^\ell \la_i = 0$.
Note that for $A \in \gl_\ell$, $[A, L]$ has $(i, j)$-entry $a_{ij}(\la_j - \la_i)$, which is zero when $i = j$ and is a nonzero multiple of $a_{ij}$ when $i \ne j$ (because $\ell \ge 3$).  In particular, $[A, L] = 0$ if and only if $A$ is diagonal.  For $x \in \gsp_{2\ell}$, we have
\[
\drho(x)v = [x,v] = \stbtmat{[A,L]}{[B,L]}{[C,L]}{[-A^\top, L]},
\]
whence $x$ is in $(\gsp_{2\ell})_v$ if and only if $A$, $B$, and $C$ are diagonal, proving the claim for $V = V(\omega_2)$.

If $\car k$ does not divide $\ell$, then $L(\omega_2) = V(\omega_2)$ and the proof is complete.  So assume $\car k$ divides $\ell$, in which case the subspace $Y_2$ of scalar matrices is a $G$-invariant submodule of $V(\omega_2)$ and $L(\omega_2) = V(\omega_2)/Y_2$.  (For this and the previous sentence, compare for example \cite[esp.~Th.~2(iv)]{PremetSup}.)  Re-reading the previous paragraph, we note that we actually proved that $[A, L]$ is scalar if and only if $A$ is diagonal, and therefore we find that the stabilizer in $\gsp_{2\ell}$ of a generic vector $v \in V(\omega_2)$ is the same as the stabilizer of its image in $L(\omega_2)$, completing the proof.
\end{proof}

When aiming to prove the existence of an s.g.p., the following lemma allows us to focus on representations that are faithful.  A shadow of it already appeared in the second paragraph of the proof of Lemma \ref{sp.wedge2}.

\begin{lem} \label{kernel}
Let $G$ be a group scheme acting on an irreducible variety $X$ such that a normal closed sub-group-scheme $N$ of $G$ acts trivially on $X$.   If there is an s.g.p.\ $(G/N)_*$ for the action of $G/N$ on $X$, then the inverse image of $(G/N)_*$ in $G$ is an s.g.p.\ for $G$ acting on $X$.
\end{lem}

\begin{proof}
The usual correspondence theorem between closed sub-group-schemes of $G$ containing $N$ and closed sub-group-schemes of $G/N$ as in \cite[Th.~5.55]{Milne:AG} shows that, for $x \in X(k)$, $(G/N)_x = G_x / N$.
For $u$ in the open subset $U$ of $X$ on which the s.g.p.\ is defined and $\gbar \in (G/N)(k)$ satisfies $\gbar (G/N)_u \gbar^{-1} = (G/N)_*$, pick $g \in G(k)$ mapping to $\gbar$.  Then $g^{-1} G_* g$ is a closed sub-group-scheme of $G$ containing $N$ with image $(G/N)_u$, so it is $G_u$.
\end{proof}

\subsection*{Reducing to a smaller problem} Suppose $V$ is a representation of an algebraic group $G$ and suppose it has an s.g.p.\ $G_*$.  Set $V_1 := V^{G_*}$, the subspace of elements fixed by $G_*$.  Then, because $G_*$ is an s.g.p., the map 
\begin{equation} \label{psi.def}
\psi \!: G \times V_1 \to V \quad \text{defined by $\psi(g,v) := gv$}
\end{equation}
is dominant, i.e., for generic $v \in V$, the orbit $G(k)v$ meets $V_1$.
In case $G_* = 1$ (i.e., $G$ acts generically freely, which is from some points of view the typical case), then $V_1 = V$ and these statements are trivial.  We remark that more sophisticated results are available, see \cite{LunaRichardson} or \cite[\S2.1]{Popov:sec} for the case $\car k = 0$ and \cite{LoetscherMacD} for results in arbitrary characteristic.

Note that the same argument works if we replace the role of $G_*$ by an s.g.p.\ $G(k)_*$ for the group of $k$-points, or $\g_*$ for the Lie algebra of $G$.  The same argument also shows that \eqref{psi.def} is dominant for $V_1 := V^\lsub$, where $\lsub$ is a subspace of $\g$ with the property that there is an open subset $U \subset V$ such that for every $u \in U$ there is a $g \in G(k)$ such that $(\Ad g)\g_u \supseteq \lsub$.

Roughly speaking, one can ``reverse'' the observation in the previous two paragraphs to find a subspace $V_1$ such that a generic $v \in V_1$ is a proxy for a generic element of $V$.  We formalize this observation in a lemma, where we write $\Trans_G(v, V_1)$ for the closed subscheme of $G$ whose $R$-points are those $g \in G(R)$ such that $gv$ is in $V_1 \otimes R$.

\begin{lem} \label{psi.sgp}
Let $V$ be a representation of an algebraic group $G$.  If a subspace $V_1$ of $V$ satisfies 
\[
\dim \Trans_G(v_1, V_1) \le \dim G + \dim V_1 - \dim V
\]
for generic $v_1 \in V_1$, then for generic $v \in V$ the orbit $Gv$ meets $V_1$.  If additionally there is a subgroup $H$ of $G$ such that $G_{v_1} = H$ for generic $v_1 \in V_1$, then $H$ is an s.g.p.\ for the action of $G$ on $V$.
\end{lem}

\begin{proof}
Suppose that $v'$ is also a generic element of $V_1$.  Then $\Trans_G(v', v_1) \subseteq \Trans_G(v', V_1)$.  For the map $\psi$ defined in \eqref{psi.def}, the fiber over $v_1$ has dimension at most $\dim G + \dim V_1 - \dim V$, i.e., the dimension of $\im \psi$ is at least $\dim V$, whence $\psi$ is dominant.
\end{proof}

\section{\texorpdfstring{$\theta$}{theta}-groups} \label{vinberg.sec}

The pairs $(G, V) = (\SL_8/\mu_4, \wedge^4 k^8)$, $(\HSpin_{16}, \text{half-spin})$, or $(\SL_9/\mu_3, \wedge^3 k^9)$ from Table \ref{meta.big} are examples of ``$\theta$-groups'' or ``Vinberg representations''.  They are constructed as follows.  Take $\Gt$ to be a split adjoint group of type $E_7$, $E_8$, or $E_8$ respectively and set $m = 2$, 2, or 3.  Pick a maximal torus $T$ in $\Gt$ and a set of simple roots $\alpha_1, \ldots, \alpha_\ell$ of $\Gt$ relative to $T$.  We define a $\Z/m$-grading on $\gt$ by setting $\gt_0$ to contain $\tor := \Lie(T)$ and $\gt_i$ to contains those root subalgebras $\gt_\alpha$ such that the height of the root $\alpha$ is congruent to $i$ mod $m$.  
We find that $G$ is a subgroup of $\Gt$ such that $\g$ is identified with $\gt_0$ and the adjoint action of $G$ on $\gt_1$ is equivalent to the representation $V$.  See \cite{Vinberg:weyl}, \cite{ABS}, \cite[\S8.5]{PoV}, \cite{Levy:Vinberg}, and \cite{RLYG} for more on this general family of representations.

If $\car k \ne m$, then it was verified in \cite[\S 7]{GG:irred} that $\g_v = 0$ for generic $v \in V$.  As the s.g.p.\ exists for the action by the group of $k$-points $G(k)$ on $V$ and $G_v$ is smooth (because $\dim \g_v = \dim G(k)_v$), the s.g.p.\ exists for the action of the group scheme $G$.

The rest of this section concerns the case $\car k = m$.

\begin{prop} \label{wedge4}
Suppose $\car k = 2$ and $(G, V)$ is either $(\SL_8/ \mu_4, \wedge^4 k^8)$ or $(\HSpin_{16}, \text{half-spin})$.  Then the s.g.p.\ exists for the action of $G$ on $V$ and is $(\Z/2)^r \times \mu_2^r$ for $r = 3$ or $4$, respectively.
\end{prop}

The $\HSpin_{16}$ case was proved in Premet's appendix to \cite{GG:spin}.  We adapt his method to encompass both cases and present it here in a side-by-side proof in order to highlight the similarities between the two cases.  The case $G = \SL_8/\mu_4$ has two minor extra complications, namely that the adjoint group $\Gt$ is not simply connected and that there is an outer automorphism of $G$ arising from conjugation by an element of $\Gt$.

\begin{proof}
The $E_7$ root system is contained in $E_8$ in the span of $\alpha_1, \ldots, \alpha_7$.
In the root system of $\Gt$, we find a subsystem of type $A_1^\ell$ with simple roots $\gamma_1, \ldots, \gamma_\ell$.  Specifically, using the notation
\[
\esevenrt{a}{b}{c}{d}{e}{f}{g}
\]
to denote the root $a \alpha_1+b\alpha_2 + c \alpha_3 + d\alpha_4 + e\alpha_5 + f\alpha_6 + g\alpha_7$ of $E_7$ and similarly for $E_8$, as in \cite[Pl.~VI and VII]{Bou:g4}.  We take $\gamma_1, \ldots, \gamma_7$ to be
\[
\esevenrt0010000, \esevenrt0100000, \esevenrt0000100, \esevenrt0112100, \esevenrt0000001, \esevenrt0112221, \esevenrt2234321   
\]
and in the case of $E_8$ we set also $\gamma_8 = \eeightrt23465432$.  (These choices agree with the ones made in the proof of Prop.~5.1.1 in \cite{GurLawther}.)  The $\gamma_i$ generate a sublattice of $T^*$, corresponding to a quotient $\Tb$.  We put $H := \ker [ T \to \Tb ]$, a finite group scheme that is the Cartier dual of $T^* / \Tb^*$, i.e., $H \cong \mu_2^r$ for $r$ as in the statement.

We can describe $H$ explicitly as follows.  As the root system of $\Gt$ is simply laced, we identify it with its inverse root system and roots with coroots.  The coroot $\gamma_i$ defines a cocharacter $\omega_i \!: \Gm \to T$ with differential $\domega_i \!: k \to \tor$ such that $\domega(1) = h_{\gamma_i}$, an element of the Cartan basis of $\gt$ corresponding to the coroot $\gamma_i$.  The composition $\gamma_j \circ \omega_i$ is either trivial ($i \ne j$) or the squaring map $(i = j)$, so $\omega_i$ embeds a copy of $\mu_2$ in $H$.  Varying $i$, we find all of $H$.  On the level of Lie algebras, $\tor$ is identified with $T^* \otimes k$ and $\lsub := \Lie(H)$ is the span of the $h_{\gamma_i}$ with $\dim \lsub = r$.  (Note that when $\Gt$ has type $E_7$, $T^*$ is the weight lattice, and $h_{\alpha_2} + h_{\alpha_5} + h_{\alpha_7} = 0$ in $\tor$.)

To determine the centralizers $C_G(H)$ and $C_G(\lsub)$, we consider instead the bilinear pairing on the roots.  For $M$ the Cartan matrix and $\eta$ the matrix whose $i$th row is the coefficients of $\gamma_i$, the product $\eta M$ viewed as a matrix with entries in $k$ has right kernel spanned by the rows of $\eta$.  That is, if $\alpha$ is in the root lattice $Q$ and $\qform{ \alpha, \gamma_i }$ is divisible by 2 for all $\gamma_i$, then $\alpha$ is in the span of the $\gamma_i$'s and $2Q$.  One finds that the roots $\alpha$ of even height do not lie in this span, i.e., have odd inner product with some $\gamma_i$.  Therefore, the 1-parameter subgroup $G_\alpha$ does not belong to either centralizer, because there is some $\gamma_i$ such that $\alpha$ and $\gamma_i$ have odd inner product.  It follows that $T$ is the identity component of $C_G(H)$ and $C_G(\lsub)$.  

Each $\gamma_i$ has odd height, and we set $V_1$ to be the $2\ell$-dimensional subspace of $\gt_1$ spanned by $\gt_{\gamma_i}$ and $\gt_{-\gamma_i}$ for all $i$.  Let $v \in V_1$ be generic.  (We remark that at this point we have observed that $\lsub = \g_v$.)  For any 2-by-2 matrix $A = \stbtmat{0}{x}{y}{0}$, $A^2$ is the scalar matrix with $xy$ on the diagonal.  It follows that the elements $v^{[2]^e} \in \lsub$ for $e \ge 1$ span $\lsub$, so the stabilizer $G_v$ centralizes $\lsub$ and normalizes $C_G(\lsub)^\circ$, i.e., $G_v$ is contained in $N_G(T)$.

For $v' \in V_1$ also generic, the transporter $\{ g \in G \mid gv = v' \}$ consists of elements normalizing $T$.  It follows that $\Trans_G(v_1, V_1) \subseteq N_G(T)$, so the transporter has dimension at most $\ell$.  For both choices of $(G, V)$ we have $\dim V = \dim G + \dim V_1 - \ell$, so Lemma \ref{psi.sgp} applies.  

We now compute the generic stabilizer $G_v$.  As $\car k = 2$, the simple reflections in the Weyl group are elements of $N_{\Gt}(T)$ of order 2, giving an expression for $N_{\Gt}(T)$ as a semidirect product of $T$ and the Weyl group.
An elementary argument as in \cite[p.~552]{GG:spin} shows that for each $i$ $\Gt_v$ contain an element whose image in the Weyl group $N_{\Gt}(T)/T$ is the reflection in the root $\gamma_i$ and that these elements account for all the cosets of $T$ in $\Gt_v$.  That is, $\Gt_v$ is a semi-direct product $H \rtimes (\Z/2)^\ell$.

The centralizer of $G$ in $\Gt$ is $G$ itself---$G$ is a maximal rank subgroup by construction.  If $\Gt$ has type $E_8$, then there is no element of $\Gt$ that normalizes $G$ and such that conjugation is an outer automorphism of $G$ because in that case we would find inside the representation $\gt$ of $G$ both half-spin representations but there is only one; in this case $N_{\Gt}(G) = G$.  If $\Gt$ has type $E_7$, then $N_{\Gt}(G) = G \rtimes \Z/2$, because $-1$ is an element of the Weyl group of $E_7$ (so is given by conjugation by an element of $N_{\Gt}(T)$) but not the Weyl group of $A_7$.

For each of the elements of $(\Z/2)^\ell$ in the Weyl group of $\Gt$ generated by the reflections in the $\gamma_i$, one checks whether it normalizes the roots of $G$.  In both cases, one finds a subgroup of order 16, i.e., $(\Z/2)^4$.
In case $\Gt$ has type $E_8$, this subgroup belongs to $G$ by the previous paragraph, proving the claim that $G_v$ is $(\Z/2)^4 \times \mu_2^4$.  In case $\Gt$ has type $E_7$, again by the preceding paragraph, one finds that $G_v$ is $(\Z/2)^3 \times \mu_2^3$.  In either case, $G_v$ does not depend on the choice of generic element $v \in V_1$.   Lemma \ref{psi.sgp} gives that $G_v$ is the s.g.p.\ for the action of $G$ on $V$.
\end{proof}

We now treat the remaining case.  The argument is similar to the preceding.  

\begin{prop} \label{wedge3}
Suppose $\car k = 3$ and $(G, V) = (\SL_9/\mu_3, \wedge^3 k^9)$.  Then the s.g.p.\ exists for the action of $G$ on $V$ and is $(\Z/3)^2 \times \mu_3^2$.
\end{prop}

\begin{proof}
We follow the outline of the proof of the previous proposition, where $\Gt = E_8$, replacing throughout the prime 2 with 3.  We find a subsystem of type $A_2^4$ with simple roots $\gamma_1, \ldots, \gamma_8$:
\[
\eeightrt10000000, \eeightrt00100000, \eeightrt00001000, \eeightrt00000100, \eeightrt01000000, \eeightrt11232100, \eeightrt00000001, \eeightrt23465431 
\]
Note that $\gamma_i, \gamma_{i+1}$ span a subsystem of type $A_2$ for $i = 1, 3, 5, 7$.  The sublattice of $T^*$ generated by the $\gamma_i$s defines a quotient $\Tb$ of $T$, and we set $H := \ker [ T \to \Tb]$.  We find $H \cong \mu_3^2$, generated by cocharacters $\alpha_1 + 2 \alpha_3$ and $\alpha_5 + 2\alpha_6$.

As in the preceding proof, $C_G(H)$ and $C_G(\lsub)$ have identity component $T$.  We set $V_1$ to be the 12-dimensional subspace of $V$ spanned by $\gt_\beta$ for $\beta$ a root in the $A_2^4$ subsystem of height congruent to $1 \bmod 3$, i.e., $\beta = \gamma_i$, $\gamma_{i+1}$, or $-\gamma_i - \gamma_{i+1}$ for $i = 1, 3, 5, 7$.  Lemma \ref{psi.sgp} applies.

As the 3-by-3 matrix of the cyclic permutation $(1\,2\,3)$ has determinant 1, it follows that the element $s_{\gamma_i} s_{\gamma_{i+1}}$ of order 3 in the Weyl group of $\Gt$ is the image of an element of order three in $N_{\Gt}(T)$.  An elementary argument as in the the preceding proof shows that $N_{\Gt}(T)/T$ is the subgroup of the Weyl group generated by these simple reflections, isomorphic to $(\Z/3)^4$.  The subgroup of this stabilizing the roots of $G$ is $(\Z/3)^2$.  It follows, therefore, that $G_v$ is isomorphic to $(\Z/3)^2 \times \mu_3^2$.
\end{proof}

\section{Infinitesimal stabilizers}  \label{inf.sec}

Table \ref{meta.big} lists 
four representations that have infinitesimal generic stabilizers.  That is, for generic $v \in V$ they have $G_v(k) = 1$ by \cite{GurLawther} and $\g_v \ne 0$ by \cite{GG:irred}.  In this section, we show that these representations have an s.g.p.\ and determine it as a group scheme.

\begin{prop} \label{inf.restrict}
If $(G, \car k, V)$ is (i) $(\Spin_5, 5, L(\omega_1 + \omega_2))$ or (ii) $(\Sp_8, 3, L(\omega_3))$, then $V$ has a s.g.p., which is isomorphic to (i) $\mu_5$ or (ii) $\mu_3 \times \mu_3$ respectively.
\end{prop}

\begin{proof}
Put $p := \car k$.  Fix a pinning for $G$, which includes a maximal torus $T$ and a Chevalley basis for $\g$.  As $G$ is simply connected, the cocharacter lattice $\Hom(\Gm, T)$ is identified with the coroot lattice for the root system of $G$.  Put $H$ for the subtorus of $T$ generated by $\im \beta^\vee$ for $\beta^\vee$ as follows:
\begin{equation} \label{inf.cor}
\begin{array}{cl@{\hskip 1in}cl}
(i)&\alpha^\vee_1 + 2\alpha^\vee_2%
&(ii)&\alpha^\vee_1 + \alpha^\vee_4,\quad  \alpha^\vee_2 + \alpha^\vee_4.
\end{array}
\end{equation}
Because $T$ normalizes $H$, it also normalizes $V^\lsub$, the subspace of $V$ annihilated by $\lsub$, and in particular $V^\lsub$ is a sum of weight spaces.  (Recall that all weights of $V$ have multiplicity 1.)  We find that $V^\lsub$ has weights:
\begin{enumerate}
\renewcommand{\theenumi}{\roman{enumi}}
\item $2\omega_1-\omega_3, -\omega_1+3\omega_2$
\item $2\omega_1-\omega_2-\omega_3+\omega_4$,  $\omega_3$, 
$\omega_1 + \omega_2 - \omega_4$, $-\omega_1 + 2\omega_2 - 2\omega_3 + \omega_4$.
\end{enumerate}
and their negatives.  

Note that $\dim G - \rank G = \dim V - \dim V^\lsub$.  As in the proof of Lemma 4.1 of \cite{GG:irred}, for generic $v \in V^\lsub$, we have $\Trans_G(v, V^\lsub) \subseteq N_G(\lsub) \subset N_G(T)$, so the transporter has dimension at most $\rank G$ and Lemma \ref{psi.sgp} applies with $V_1 := V^\lsub$.  Moreover, there is a dense open subset $U$ of $V^\lsub$ such that $G_u$ is a closed subgroup of $N_G(T)$ for each $u \in U(k)$.  On the other hand, for generic $v \in V$, $G_v(k) = 1$ by \cite[\S2.7]{GurLawther}, which shows that (after possibly shrinking $U$), $G_u$ is a closed subgroup of $T$ for each $u \in U(k)$.

That is, $G_u$ is the intersection of $\ker \omega\vert_T$ as $\omega$ varies over weights displayed above.  The quotient of the weight lattice $T^*$ by the sublattice generated by those weights is (i) $\Z/5$ or (ii) $\Z/3 \times \Z/3$ respectively, proving the claim.
\end{proof}

\begin{prop} \label{SL4.cases}
Let $(G, \car k, V) = (\SL_4, p, L(p^e \omega_1 + \omega_2))$.  If 
\begin{enumerate}
\renewcommand{\theenumi}{\roman{enumi}}
\item $p$ is odd and $e \ge 1$ or
\item $p = 2$ and $e \ge 2$,
\end{enumerate}
then $V$ has an s.g.p., which is isomorphic to (i) $\mu_{p^e}$ or (ii) $\mu_{p^{e+1}}$ respectively.
\end{prop}

\begin{proof}
Put $q := p^e > 1$.  Identify the natural representation $k^4$ with $L(\omega_1)$.   Put $W_1$
(resp.~$W_2)$ for the subspace of vectors with the last (resp.~first)
two coordinates zero, so $k^4 = W_1 \oplus W_2$.  Put $\hat{J}$ for
the subgroup of $\SL_4$ normalizing or interchanging $W_1$ and $W_2$, equivalently, the normalizer of $W_1 \otimes W_2$ in $L(\omega_2) = \wedge^2 (W_1 \oplus W_2)$.  Its
identity component consists of block diagonal matrices with diagonal
elements $g_1, g_2 \in \GL_2(k)$ such that $\det g_1 \det g_2 = 1$.
Put 
\[
U := (W_1 \oplus W_2)^{[q]} \otimes (W_1 \otimes W_2) \quad \subseteq L(\omega_1)^{[q]} \otimes L(\omega_2) = V.
\]

We verify the hypotheses of Lemma \ref{psi.sgp}, with $V_1 := U$.  A generic element of $U$ is $u = \sum x_i \otimes y_i$ with the $y_i$ a basis for $W_1 \ot W_2$.
If $gu'=u$ for some $u' \in U$, then $g$ preserves $W_1 \ot W_2$ and so the dimension
of the transporter is at most 
\[
\dim \hat{J} = 7 = \dim G + \dim V_1 - \dim V,
\]
as required.  We know by \cite[Prop.~2.8.3]{GurLawther} that the stabilizer $H$ of a generic $v \in U$ is infinitesimal, and in particular is connected and contained in the identity component $J$ of $\hat{J}$.

Let $T$ denote the copy of $\Gm$ in $J$ such that $t \in \kx$ acts on $W_1$ as multiplication by $t$ and on $W_2$ as multiplication by $t^{-1}$.  The group $J$ is evidently generated by $T$ and the subgroup $\SL_2 \times \SL_2$ of block diagonal matrices in $\SL_4$ mentioned in the first paragraph of the proof.  The two groups overlap in a copy of $\mu_2$ (diagonally embedded in $\SL_2 \times \SL_2$) and we find that $J$ is $(T \times \SL_2 \times \SL_2)/\mu_2$.  

For $R$ a commutative $k$-algebra write $[a, b, c]$ for the image of $(a, b, c) \in T(R) \times \SL_2(R) \times \SL_2(R) \in J(R)$.  We remark that there are two ``obvious'' copies of $\mu_4$ in $J$, namely a $\mu_4 < T$ and the center of $\SL_4$.
Indeed, suppose $\zeta \in \mu_4(R)$ has order 4.  The element $[\zeta, 1, 1] \in J$ belongs to $\mu_4(R) \subset T(R)$, whereas $[\zeta, \zeta^2, 1] = [\zeta^3, 1, \zeta^2]$ belongs to the center of $\SL_4$.

\medskip
\emph{The kernel $K$ of the action.}
Let us determine the group scheme $K := \ker [J \to \GL(U)]$.  If $[a,b,c] \in K(R)$, then $b, c$ belong to $Z(\SL_2)(R) = \mu_2(R)$ and in particular $c^2 = 1$.  If $c \ne 1$, then we have $[a, b, c] = [a, b, c][c, c,c] = [ac, bc, 1]$ in $J$, so every element in $K(R)$ is of the form $[a, b, 1]$ for some $b \in \mu_2(R)$.

On $W_1^{[q]} \otimes W_1 \otimes W_2$, the element $[a, b, 1]$ acts by a scalar $a^q b^{q+1}$.  For the other summand in $U$, the scalar is $a^{-q} b$.  So $[a, b, 1]$ is in $K(R)$ if and only if
\begin{equation} \label{SL4.1}
   a^q b^{q+1} = 1 = a^{-q} b.
\end{equation}
In this way, we find a copy of $\mu_q$ contained in $K$ as elements $[a, 1, 1]$ for $a \in \mu_q(R)$.  This subgroup contains all the elements of $K(R)$ of the form $[a, 1, 1]$.

If $p$ is odd, then \eqref{SL4.1} reduces to $a^q = 1 = b a^{-q}$.  That is, $b = 1$ and we find that $K = \mu_q \subset T$.

If $p = 2$, then \eqref{SL4.1} reduces $a^q b = 1 = a^{-q} b$, so
  $a^{2q} = 1$   and $b = a^q$.
That is, there is an isomorphism $\mu_{2q} \to K$ defined on $R$-points via
$a \mapsto [a, a^q, 1] = [a^{q+1}, 1, a^q]$ and $K \cap T = \mu_q$.

Alternatively, consider a maximal torus $S := T \times \Gm \times \Gm \subset T \times \SL_2 \times \SL_2$ where $\Gm$ stands for the diagonal matrices in $\SL_2$.  The image of $S$ in $J$ contains the kernel $K$.  Writing $(1, 0, 0)$ for a fundamental weight on $T$ and similarly for the other two components of $S$, the weights of $U$ are
\begin{equation} \label{SL4.wts}
(q, \pm q \pm 1, \pm 1) \quad \text{and} \quad (-q, \pm 1, \pm q \pm 1)
\end{equation}
where the signs may be chosen independently.  The kernel of $S$ on $U$ is the intersection of the kernels of these weights; to say the same thing differently, the image $\Sb$ of $S$ in $J/K$ has character lattice $\Sb^*$ in $S^*$ generated by the weights of $U$.  When $p = 2$, the lattice has index $4q$ in $S^*$ and basis $(q, 1, 1)$, $(0, 2, 0)$, $(0, 0, 2)$.

\medskip
\emph{The center of $\Lie(J/K)$.}  Put $\Tb$ for the central torus in $J/K$, the image of $T$.  If $p$ is odd, then the center of $\SL_2 \times \SL_2$ injects into $J/K$, and we find that $Z(J/K) \cong \Tb \times \mu_2$.  If $p = 2$, then the image of the center of $\SL_2 \times \SL_2$ in $J/K$ is a copy of $\mu_2$, which is contained in $\Tb$, so $Z(J/K) = \Tb$.  In either case, $\Lie(Z(J/K)) = \Lie(\Tb)$.

\medskip
\emph{Verification that $H = K$.}  Trivially the generic stabilizer $H = J_v$ contains $K$, and we claim they are equal.  To see this, we verify that $J/K$ acts generically freely on $U$.  Indeed, $H/K = (J/K)_v$, and $(H/K)(k) = 1$ because $H(k) = 1$ \cite[Prop.~5.47]{Milne:AG}.  Consequently, we are reduced to showing that $\Lie(J/K)_v = 0$.

Suppose for the moment that
\begin{equation} \label{SL4.claim}
\dim U^x \le 8 \quad \text{for nilpotent or semisimple $x \in \Lie(J/K) \setminus \Lie(\Tb)$.}
\end{equation}
As $\dim (\Ad J/K)x \le \dim J/K - \rank J/K = 4$,
we find that
\[
\dim (\Ad J/K)x + \dim U^x < \dim U,
\]
whence $\Lie(J/K)_v = \Lie(\Tb)_v$ by \cite[Lemma 1.6(2)]{GG:large}.  As $J/K$ is reductive and $\Lie(J/K)_v$ is contained in the center of $\Lie(J/K)$, it follows easily that $\Lie(J/K)_v = 0$, compare \cite[Lemma 1.7]{GG:large}.

\medskip
\emph{Verification of \eqref{SL4.claim}.}  We now go back and verify claim \eqref{SL4.claim}.  Suppose first that $x$ is semisimple.  There is a maximal torus of $J/K$ whose Lie algebra contains $x$ \cite[Th.~13.3, Rem.~13.4]{Hum:p}; since all maximal tori are conjugate we may assume that $x$ is in $\Lie(\Sb)$.  If $p$ is odd, then $\Lie(J/K) \cong \tor \oplus \so_4$ (because $T/\mu_q \cong T$) where $t \in \tor$ acts on $W_i^{[q]} \ot W_1 \ot W_2$ via $(-1)^{i+1} t$ and $\so_4$ acts on it as a sum of 2 copies of its natural representation.  Write $x = t + x_0$ for $t \in \tor$ and $x_0 \in \so_4$.  As $x$ is not central, $x_0 \ne 0$, so the largest eigenspace  of $x_0$ on the natural representation has dimension at most 2.  It follows that $\dim (W_i^{[q]} \ot W_1 \ot W_2)^x \le 4$, whence \eqref{SL4.claim}.

Suppose now that $x$ is semisimple and $p = 2$.  For each weight of $U$ on $S$ as in \eqref{SL4.wts}, we express it in terms of the basis for the sublattice $\Sb^*$ of $S^*$ and reduce mod 2 to find the weights of $U$ on $\Lie(\Sb)$; these are 
\[
(1, 0, 0), \quad (1, 0, 1), \quad (1, 1, 0), \quad (1, 1, 1),
\]
each with multiplicity 4.  If $\dim U^x > 8$, then at least three of these vanish on $x$.  But any three of these are linearly independent, so $x = 0$, a contradiction, verifying \eqref{SL4.claim}.

Suppose now that $x$ is nilpotent.  Then $x$ is the image of an element in 
\[
\so_4 = \Lie((\SL_2 \times \SL_2)/\mu_2) \subset J/\mu_q,
\] and $x$ acts on $U$ as on 4 copies of the natural representation $k^4$.  Since a nonzero nilpotent in $\so_4$ has kernel of dimension at most 2 on the natural module, \eqref{SL4.claim} follows.
 \end{proof}
 
In the previous proposition when $p = 2$, $\SL_4$ does not act faithfully on $V$.  To address this, we provide: 
 \begin{cor}
If $(G, \car k, V) = (\SL_4/\mu_2, 2, L(2^e \omega_1 + \omega_2))$ for some $e \ge 2$, then $V$ is a faithful representation of $G$ and has an s.g.p.\ that is isomorphic to $\mu_{2^{e}}$.
\end{cor}

\begin{proof}
$\SL_4$ acts on $V$ with kernel $\mu_2$ and generic stabilizer isomorphic to $\mu_{2^{e+1}}$ (Prop.~\ref{SL4.cases}), so Lemma \ref{kernel} gives that $\SL_4/\mu_2$ has an s.g.p.\ and it is isomorphic to $\mu_{2^{e+1}} / \mu_2 \cong \mu_{2^e}$.
\end{proof}

\section{Existence of an s.g.p.: Proof of Theorem \ref{sgp}}  \label{sgp.pf}

We will now prove Theorem \ref{sgp}, where $V$ is an irreducible representation of a simple algebraic group $G$.
Suppose first that $G$ acts faithfully on $V$.   The case $V = L(\hst)$ was treated in Lemma \ref{adj.eg}.
 If $G$ acts generically freely on $V$, then there is nothing to prove, so assume $G_v \ne 1$ for generic $v$.
 
\smallskip
Consider the case $\dim V > \dim G$.  By Theorem \ref{summ.thm}\eqref{summ.big}, we may assume that $(G, \car k, V)$ is listed in Table \ref{meta.big}.

In the cases where $G_v$ is finite \'etale for generic $v$, \cite{GurLawther} shows that the s.g.p.\ exists for the action by $G(k)$.  Since $\g_v = 0$ for generic $v$, it follows that there is an s.g.p.\ for the action of $G$ (Lemma \ref{smooth.sgp}).

For ($\HSpin_{16}$, any, half-spin) see \cite[Th.~1.2]{GG:spin} or Prop.~\ref{wedge4}.  For $(\SL_8/\mu_4, 2, \wedge^4 k^8)$ and $(\SL_9/\mu_3, 3, \wedge^3 k^9)$, see Propositions \ref{wedge4} and \ref{wedge3} respectively.  In all three cases, $G_* = (\Z/p)^r \times \mu_p^r$ for $p := \car k$ and some $r > 1$, so the identity component of $G_*$ is $\mu_p^r$.  The proofs in \S\ref{vinberg.sec} show that in each case $\mu_p^r$ is contained in a torus of $G$.

The four cases in Table \ref{meta.big} where $G_v(k) = 1$ were treated in \S\ref{inf.sec}.  In each case, $G_*$ is connected and contained in a maximal torus of $G$.

We note that in all of these cases with $\dim V > \dim G$, the the s.g.p.\ $G_*$ is a finite group scheme.  For the third claim, since $G_*$ is finite, its identity component $G_*^\circ$ is non-trivial exactly when $G_*$ is not smooth; those cases are covered in the two preceding paragraphs.  Since $G_*^\circ$ is contained in a maximal torus $T$ of $G$, we have
$\Lie(G_*) = \Lie(G_*^\circ) \subseteq \Lie(T)$.
 
\smallskip
Consider now the cases in Table \ref{meta.small}.
If $\dim V/G \le 1$, then there is an open $G$-orbit in $\P(V)$ \cite[\S6]{BGL}, hence an s.g.p.\ exists.
For ($\PSp_{2\ell}$, any, $L(\omega_2)$), the existence of an s.g.p.\ was established in Lemma \ref{sp.wedge2}.
The s.g.p.\ for the representation ($\Spin_{13}$, any, spin) was calculated in \cite[\S8,9]{GG:spin} and \cite[Prop.~5.2.9]{GurLawther}, see Remark \ref{Spin13}.  

There is one final case from Table \ref{meta.small}, which we treat in the following lemma.

\begin{lem}[natural representation of $F_4$] \label{F4}
Let $G$ be a group of type $F_4$, and let $V$ be its ``natural'' Weyl module $V(\omega_4)$ of dimension 26 or the irreducible quotient $L(\omega_4)$ of the Weyl module.  Then the s.g.p.\ exists for $G$ acting on $V$, and it is isomorphic to $\Spin_8$.
\end{lem}

\begin{proof}
Suppose first that $V$ is the Weyl module.  The group $G$ can be viewed as the (algebraic) group of automorphisms of an Albert algebra $J$, where $V$ is the codimension-1 subspace of elements of trace zero.  See \cite{Ptr:surv} and \cite{Jac:J} for background on Albert algebras.

As a vector space, $J$ can be written as the set of hermitian 3-by-3 matrices with entries in an octonion algebra.  The set $E$ of diagonal matrices is a cubic \'etale algebra, and the sub-group-scheme of $G$ fixing $E$ elementwise is isomorphic to $\Spin_8$, by the arguments in \cite[\S38]{KMRT} and \cite[Th.~6]{Jac:ex}.  (Although stated under the hypothesis that $\car k \ne 2, 3$, the arguments go through without this hypothesis with only cosmetic changes.  Alternatively, in case $\car = 2$ or 3, one can check this on the level of $k$-points and use a computer to verify that the generic stabilizer in $\g$ has dimension at most 28.)

A generic element $v \in V$ generates a cubic \'etale subalgebra $E_v$ of $J$.  (This is essentially when $\car k \ne 2$, because in that case every element generates a commutative associative subalgebra, and properties of the generic minimal polynomial show that a generic element generates a separable subalgebra of degree 3.  When $\car k = 2$, the same reasoning works by invoking \cite[Prop.~1]{McC:FSTrev}.)
Therefore the stabilizer $G_v$ of $v$ is the subgroup of $G$ fixing $E_v$ elementwise.  As $k$ is algebraically closed, $E_v$ is isomorphic to $E$ as a $k$-algebra, and so is conjugate under $G$ to $E$ by Jacobson's Embedding Theorem (which, in this generality, is \cite[\S4.12]{Ptr:embed}).  This proves the claim for the Weyl module.

The Weyl module only fails to be irreducible when $\car k = 3$ \cite{luebeck}.  In that case, the irreducible representation is the quotient of $V$ by the span of the identity element of $J$.  For $v$ generic in the irreducible quotient, any inverse image of it in the Weyl module again generates a cubic \'etale subalgebra of $J$, and the Weyl module case again shows that the s.g.p.\ exists for the action of $G$.
\end{proof}

 We have now accounted for all of the representations in Table \ref{meta.small}, proving Theorem \ref{sgp} under the assumption that $V$ is faithful.
So drop the assumption that $V$ is faithful and put $N := \ker [ G \to \GL(V) ]$.  Then $G/N$ is simple (Lemma \ref{simple.quo}) and acts faithfully on $V$, so it has an s.g.p.\ $(G/N)_*$.  The inverse image of $(G/N)_*$ in $G$ is an s.g.p.\ for the action of $G$ (Lemma \ref{kernel}).  If $\dim V > \dim G$, then since $V$ is irreducible, $G$ acts nontrivially on $V$ and $N$ is a finite group scheme.
We conclude that $\dim G_* = \dim N + \dim (G/N)_* = 0$, i.e., $G_*$ is also finite.  If additionally $N$ is central, then $N$ is contained in every maximal torus $T$ of $G$ and there is a bijection between maximal tori of $G$ and $G/N$ given by $T \leftrightarrow T/N$ \cite[Th.~2.20(ii)]{BoTi:c}.  Therefore, the inverse image in $G$ of any maximal torus of $G/N$ containing $(G/N)_*$ is a maximal torus of $G$ containing $G_*$.
This completes the proof of Theorem \ref{sgp}.

\section{The s.g.p.\ is commutative for large \texorpdfstring{$V$}{V}}  \label{commutative.sec}

Building on what has gone before, we easily obtain the following result concerning when the s.g.p.\ is commutative.

\begin{prop} \label{commutative}
Let $V$ be a faithful and irreducible representation of a simple algebraic group $G$.  If $\dim V > \dim G + 1$, then either
\begin{enumerate}
\item \label{comm.1} for generic $v \in V$ the stabilizer $G_v$ is a commutative group scheme or
\item \label{comm.2} $(G, \car k, V) = (\SL_9/\mu_3, 2, \wedge^3 k)$, up to graph automorphism.
\end{enumerate}
\end{prop}

\begin{proof}
Assume $G_v \ne 1$, for otherwise there is nothing to prove.
By Theorem \ref{summ.thm}\eqref{summ.big}, up to graph automorphism 
$(G, \car k, V)$ belongs to Table \ref{meta.big}.  In that table, only four rows have $G_v$ non-commutative.  (This claim relies on the results of \S\ref{vinberg.sec} and \S\ref{inf.sec}.)  Of those four, only the row $(\SL_9/\mu_3, 2, \wedge^3 k^9)$ has dimension at least $\dim G + 1$.
\end{proof}

\section{Smoothness: Proof of Theorem \ref{stab.smooth}}  \label{smooth.sec}

We now address the question of whether the group scheme $G_v$ stabilizing a generic $v \in V$, is smooth.  Specifically, we prove Theorem \ref{stab.smooth}.

We may assume that the stabilizer $G_v$ of a generic $v \in V$ is not the trivial group scheme.
If $V = L(\hst)$, then the generic stabilizer has identity component a maximal torus (Lemma \ref{adj.eg}).
Therefore, by Theorem \ref{summ.thm}, we may assume that, up to graph automorphism, $(G, \car k, V)$ belongs to Table \ref{meta.small} or \ref{meta.big}.

We complete the proof of Theorem \ref{stab.smooth} by noting that the representations in Table \ref{meta.small} have $G_v$ smooth except for $G_2$ in characteristic 2.  For $\SO_n$ with $n \ge 5$, when $n$ is odd or $\car k \ne 2$, the stabilizer of an anisotropic vector in the tautological representation $k^n$ is $\SO_{n-1}$.
The stabilizers in spin or half-spin representations of groups of type $B$ and $D$ are determined in \cite{GG:spin}.

The natural representation of $F_4$ is treated in Lemma \ref{F4}.  (For $\car k \ne 3$, the smoothness was established by different means in \cite[Th.~3.2]{Stewart:min}.)  One can check easily the cases of the natural representations of $\SL_n$ and $\Sp_{2\ell}$.  For $\SL_n$ acting on $S^2 k^n$ with $\car k \ne 2$  or $\wedge^2 k^n$ with $n$ even, the stabilizer of a generic element is the definition of the special orthogonal or symplectic group.  If $n$ is odd for $\wedge^2 k^n$, it is the stabilizer of
a degenerate form but it is straightforward to compute that it is smooth.

\medskip

Here are two techniques that leverage our knowledge of the stabilizer $G(k)_v$ on the level of $k$-points.  First, $G_v$ is smooth if and only if $\dim \g_v = \dim G(k)_v$.  On the other hand, we may view $\g$ as a representation of the abstract group $G(k)_v$, i.e., the $k$-points of the smooth subgroup $(G_v)_{\red}$ of $G_v$. The group $G(k)_v$ normalizes the Lie algebra $\g_v$, which contains $\lsub := \Lie((G_v)_{\red})$ as a $G(k)_v$-invariant subalgebra.  Moreover, $\dim \lsub = \dim G(k)_v$.   By hypothesis, $\g$ acts faithfully on $V$, so $\g_v \ne \g$ (otherwise $\g$ would act trivially) and $\z(\g) \cap \g_v = 0$ (if a nonzero central element of $\g$ annihilates a nonzero vector in $V$, it annihilates $V$).  This provides many constraints on $\g_v$.

Second, for any particular characteristic $p = \car k$, we may construct $\g$ and $V$ over $\F_p$.  For every field $k' \supseteq \F_p$ and every $v' \in V \ot k'$, we have $\dim \g_v \le \dim \g_{v'}$.  In particular, it suffices to find such a $v'$, say with $k' = \F_{p^3}$, such that $\dim \g_{v'} = \dim G(k)_v$; such a $v'$ can be sought using a computer as described in \cite[\S3]{GG:irred}.

\medskip

The next result concerns the ``natural'' representation of a group $G$ of type $G_2$, which has highest weight $\omega_2$.  The Weyl module $V(\omega_2)$ is 7-dimensional and has a nonzero $G$-invariant quadratic form $q$, see for example \cite[\S4.4]{GN}.  One can argue as in \cite[\S23]{FH} or \cite[1.6.4, 1.7.3, 2.2.4]{Sp:ex} that the orbit of the highest weight vector is the set of nonzero $v$ such that $q(v) = 0$.  The stabilizer of $v$ in $G$ has codimension 1 in a parabolic subgroup, and we denote it by $A_1 U_5$, where $U_5$ stands for the unipotent radical of the parabolic subgroup. The representation $V(\omega_2)$ is irreducible if and only if $\car k \ne 2$.

\begin{lem}[natural representation of $G_2$] \label{G2}
For $G$ of type $G_2$ and $V = L(\omega_2)$, the stabilizer $G_v$ is smooth (resp.~reductive) for generic $v \in V$ if and only if $\car k \ne 2$.  If $\car k = 2$, then $G(k)$ acts transitively on the nonzero vectors in $V$, $G_v$ is not smooth, and $(G_v)_\red = A_1 U_5$.
\end{lem}

\begin{proof}
Over $\Z$, we take an ordered basis of the 7-dimensional Weyl module $V(\omega_2)$ consisting of weight vectors $v_\mu$ for the weights $\mu = 2\alpha_1 + \alpha_2$, $\alpha_1 + \alpha_2$, $\alpha_1$, $0$, $-\alpha_1$, $-(\alpha_1+\alpha_2)$, $-(2\alpha_1+\alpha_2)$ respectively and such that the root element $e_{\alpha_1}$ and corresponding 1-parameter subgroup $x_{\alpha_1} \!: \Ga \to G_2$ corresponding to the root $\alpha_1$ are given by the matrices
\begin{equation} \label{G2.1}
e_{\alpha_1} = 
\left(
\begin{smallmatrix}
0&-1 \\
&0 \\
&&0 & 1 &  \\
&&&0&2 \\
&&&&0 \\
&&&&&0&-1 \\
&&&&&&0 
\end{smallmatrix}
\right)
\quad \text{and} \quad
x_{\alpha_1}(t) =  \exp(t e_{\alpha_1})=
\left(
\begin{smallmatrix}
1&-t \\
&1 \\
&&1 & t & t^2 \\
&&&1&2t \\
&&&&1 \\
&&&&&1&-t \\
&&&&&&1 
\end{smallmatrix}
\right)
\end{equation}
as in the proof of \cite[Prop.~5.2.14]{GurLawther}.

If $\car k \ne 2$, then $V$ is the base change to $k$ of the Weyl module, and the stabilizer $G_v$ is $\SL_3$.  This can be seen from the further explicit calculations in the proof of \cite[Prop.~5.2.14]{GurLawther} or as in \cite{Stewart:min} or
by identifying $V$ with the space of trace zero octonions as in \cite[p.~507, Exercise 6c]{KMRT}.

If $\car k = 2$, then $V$ is obtained from the Weyl module over $k$ by
modding out by the span of $v_0$.  The action of $G$ on $L(\omega_2)$
preserves the alternating bilinear form obtained from $q$ on
$L(\omega_2)$ and so gives an inclusion $G \to \Sp_6$.  For any finite
field $K$ of characteristic 2, one has  $\Sp_6(K) = G(K) \Sp_6(K)_v$
for any nonzero $v \in L(\omega_2)$ \cite{LPS:fact}.
The same factorization of $\Sp_6$ therefore holds over the algebraic
closure of $\F_2$ and so over the algebraically closed field $k$ of
characteristic 2.  The transitivity of the action for $G(k)$ now
follows from the same transitivity for $\Sp_6(k)$.

Here is an alternate argument that $G(k)$ acts transitively.
For each $y \in V$, pick $x \in V(\omega_2)$ such that $x \mapsto y$.  One can argue (e.g., by interpreting the Weyl module as the trace zero subspace of the octonions) that the $G$-invariant quadratic from $q$ on the Weyl module is not zero on $v_0$, so by scaling $q$ we may assume that $q(v_0) = 1$.  For each $\lambda \in k$, we have $q(x + \la v_0) = q(x) + \la^2$, so there is a unique choice of $x$ such that $x \mapsto y$ and $q(x) = 0$.  Since $G(k)$ has two orbits on the hypersurface $q = 0$, it has two orbits on $L(\omega_2)$.    Note that this argument shows that, on the level of $k$-points, the stabilizer agrees with the stabilizer of the highest weight vector, i.e., is $A_1 U_5$.  (Or see \cite[Prop.~5.2.14]{GurLawther}.)

It remains to verify that the stabilizer is not smooth.
One can read off the action of elements of a Chevalley basis on $V$ by writing them as matrices as we have done above and deleting the 4th row and column.  By the transitivity of  $G(k)$, we may pick the highest weight vector $v = v_{2\alpha_1 + \alpha_2}$ as a generic vector.  The Lie algebra stabilizer $\g_v$ is normalized by the maximal torus $T$ in $G$ underlying these calculations (because $Tv \subseteq kv$), so $\g_v$ is a sum of $\g_v \cap \tor$ and those root subalgebras $\g_\alpha$ that belong to $\g_v$.
We note from \eqref{G2.1} that
$e_{\alpha_1}$ annihilates the image of $v_{-\alpha_1}$ in $V$.  Since $\alpha_1$ and $2\alpha_1 + \alpha_2$ are both short roots, they are in the same orbit under the Weyl group, and it follows that $e_{-(2\alpha_1 + \alpha_2)}$ annihilates $v$, so $\dim \g_v \ge 9$.  One can check that this is all of $\g_v$ by verifying that each of the remaining four root subalgebras $\g_\alpha$ for negative $\alpha$ do not annihilate $v$ or by using a computer to find a vector $v' \in V$ such that $\dim \g_{v'} = 9$.
\end{proof}

We remark that in the proof above one can read off from \eqref{G2.1} that the stabilizer of $v_{-\alpha_1}$ in $\im x_{\alpha_1}$ has $R$-points $\{ x_{\alpha_1}(t) \mid \text{$t \in R$ such that $2t = 0$} \}$ and Lie algebra the root subspace $\g_{\alpha_1}$.  (The isomorphism class of this group scheme is generally denoted $\galpha_2$.)  In this way, we can concretely see the source of the extra dimension in $\g_v$.

\begin{lem}[$\wedge^3 k^7$]  \label{w3k7}
For $G = \SL_7$ and $V = \wedge^3 k^7$, the stabilizer $G_v$ of a generic vector $v \in V$ is a simple algebraic group of type $G_2$.  In particular, $G_v$ is smooth.
\end{lem}

\begin{proof}
In the case $k = \C$, this result goes back at least to \cite{Engel}, see \cite{Agricola} for context.  For general $k$, \cite[Prop.~5.2.17]{GurLawther} shows that $G(k)_v$ are the $k$-points of a subgroup of type $G_2$ so that the tautological representation of $\SL_7$ restricts to the Weyl module $V(\omega_2)$ of $G_2$.  In the notation established above, $\dim \lsub = 14$.  It remains to show that $G_v$ is smooth, i.e., that $\g_v = \lsub$, equivalently that $\dim \g_v = 14$.

If $\car k \ne 2, 7$, then, as a representation of $G_2$, $\sl_7 = \so_7 \oplus L(2\omega_2)$, where $\so_7$ can be identified with skew-symmetric matrices and $L(2\omega_2)$ with the trace zero symmetric matrices, and $\so_7/\lsub$ is the natural representation of $G_2$.  The only Lie algebra lying between $\lsub$ and $\sl_7$ is $\so_7$.  However, the restriction of $V$ to $\so_7$ has head the spin representation with generic stabilizer $G_2$, so the stabilizer in $\so_7$ of a generic vector in $V$ can be no larger than $\lsub$, whence $\g_v \ne \so_7$, completing the proof in this case.

If $k$ has characteristic 2 or 7, one uses a computer to find a $v' \in V$ such that $\dim \g_{v'} = 14$.

(Alternatively, \cite[Prop.~5.2.17]{GurLawther} shows that an s.g.p.\ exists on the level of $k$-points.  Then, for any specific choice of $\car k$ --- whether 2, 7, or something else --- it suffices by Lemma \ref{smooth.sgp} to use a computer to find a $v \in V$ such that $\dim \g_v = 14$.)
\end{proof}

Some unusual features of Lie algebras of groups of type $G_2$ when $\car k$ is 2, 3, or 7 are discussed in \cite{CRElduque}.

For $(G, V) = (\SL_8, \wedge^3 k^8)$, one needs to show that $\dim \g_v \le  8$.  This can be argued in a manner similar to Lemma \ref{w3k7}.  Alternatively, the stabilizer $\g_v$ is computed explicitly in \cite[pp.~87--90]{SK}.  The latter proof goes through if $\car k \ne 2, 3$.  In the remaining characteristics, $\dim \g_v \le 8$ can be verified by computer. 

\smallskip

Consider now the case $(G, V) = (E_6$, minuscule).  As a representation of $F_4$, $\lieE_6 / \lieF_4$ is the smallest nontrivial Weyl module of $F_4$, $V(\omega_4)$, of dimension 26.  (See \cite{ChevSchaf} for a view of this statement from the perspective of Jordan algebras.)  The image of $\g_v$ in $V(\omega_4)$ is contained in the radical.  If $\car k \ne 3$, then this radical is zero and $\g_v = \lsub$.  If $\car k = 3$, then the radical is $\z(\lieE_6)$, and again the image of $\g_v$ is zero.  (For this case and the case of ($E_7$, minuscule) in the following paragraph, smoothness of the generic stabilizer is also contained in \cite{Stewart:min}.)

The representations $(G, V) = (\SL_6/\mu_3, \wedge^3 k^6)$, $(\Sp_6$, ``spin''), and ($E_7$, minuscule), all belong to a family of representations considered in \cite{Roe:extra} and \cite[\S12]{G:lens}.  In each case, there is a group $\Gt$ and a simple root $\alpha$ of $\Gt$ that is the only one not orthogonal to the highest root of $\Gt$, $G$ is a subgroup of $G$ generated by the roots of $\alpha$-height zero, and the root subalgebras of $\alpha$-height 1 in $\gt$ span a $G$-submodule of $\gt$ equivalent to $V$.  Moreover, there is a unique simple root $\beta$ of $G$ not orthogonal to the highest weight of $V$, and $\beta$ has coefficient 1 in the highest root of $G$ so that the root subgroups of $G$ corresponding to roots of $\beta$-height zero generate a group $G_0$ of the same description as $G(k)_v$.
If $\car k \ne 2$, the $G$-orbits in $V$ are described in terms of the root system of $\Gt$ in \cite{Roe:extra} and a representative generic vector $v$ is provided such that $G_0(k) \subseteq G(k)_v$, whence equality.  Grading $\g$ by $\beta$-height, we find $\g$ contains $\g_0$, a rank 1 torus $\tor$, and subspaces $\g_1$, $\g_{-1}$ spanned by roots of $\beta$-height 1 and $-1$ respectively; these latter two subspaces are irreducible representations of $G_0$ \cite[Th.~2c]{ABS}.  That is, the composition series of $\g_v / \lsub$, as a representation of $G_0(k)$, has simple factors contained in $\tor$, $\g_1$, $\g_{-1}$.  The explicit description of $v$ from \cite{Roe:extra} shows that these cannot be contained in $\g_v$ as in \cite[12.2]{G:lens}, whence $G_v$ is smooth.  If $\car k = 2$, we verify that $G_v$ is smooth using Magma.

The case $(G, V) = (\PSp_{2\ell}, L(\omega_2))$ for $\ell \ge 3$ has a smooth s.g.p.\ by Lemma \ref{sp.wedge2}.
This completes the proof of Theorem \ref{stab.smooth}.

\section{Representations with ``few'' invariants: proof of Theorem \ref{kVG}} \label{few.sec}

We now classify those irreducible representations $V$ such that the $\dim V/G < \dim G$; there are relatively few. 

\begin{proof}[Proof of Theorem \ref{kVG}]
If $\dim V \le \dim G$, then $V = L(\hst)$ or $(G, \car k, V)$ belongs to Table \ref{meta.small} by Theorem \ref{summ.thm}\eqref{summ.eq} and \eqref{summ.small}, so assume $\dim V > \dim G$.  Then, by  Theorem \ref{summ.thm}\eqref{summ.big}, the stabilizer of a generic vector in $V$ is a finite group scheme, whence
\begin{equation}
\dim k[V]^G = \dim V - \dim G.
\end{equation}
That is, we are reduced to determining the faithful irreducible representations of $G$ such that $\dim G < \dim V < 2 \dim G$.  This is done in the next proposition, completing the proof of Theorem \ref{kVG}.
\end{proof}

\begin{prop}
Let $V$ be a faithful and irreducible representation of a simple algebraic group $G$ over an algebraically closed field $k$.  If $\dim G < \dim V < 2 \dim G$, then up to graph automorphism $(G, \car k, V)$ appears in Table \ref{meta.big} or \ref{meta.few}.
\end{prop}

\begin{proof}
Suppose first that the highest weight $\la$ of $V$ is restricted.  Then the tables in \cite{luebeck} list all possibilities for $(G, \car k, V)$, completing the proof in this case.

If $\car k = 0$, then that is the only case, so suppose $p := \car k \ne 0$.  

We use the following observation.  Put $m$ for the smallest dimension of a nontrivial irreducible restricted representation of $G$.  We note that $m^2 > 2\dim G$ unless $G$ has type $A_\ell$ or $C_2$ and that $m^3 > 2 \dim G$ regardless of the type of $G$.

Write $\la = \sum_{e \ge 0} p^e \la_e$ where $\la_e$ is restricted for all $e$, so $\dim V = \prod_e \dim L(\la_e)$.  Because $V$ is faithful and not restricted, $\la_0 \ne 0$ and $\la_e \ne 0$ for some $e > 0$.  If a third summand is not zero, then $\dim V \ge m^3 > 2 \dim G$, and we conclude that $\la = \la_0 + p^e \la_e$ for some $e > 0$ with $\la_0, \la_e \ne 0$.

For $G$ of type $A_\ell$, $m = \ell + 1$, and the representations of this dimension are $L(\omega_1)$ and $L(\omega_1)^* = L(\omega_\ell)$.  The representations $L(\omega_1 + p^e \omega_\ell)$ and $L(\omega_1 + p^e \omega_1)$ appear in Table \ref{meta.big}.  Similarly, for $G$ of type $C_2$, $m = 4$, corresponding to $L(\omega_1)$ and the representation $L(\omega_1 + p^e \omega_1)$ appears in Table \ref{meta.few}.

So suppose that $\dim L(\la_i) > m$ for $i = 0$ or $e$.  Noting that $\dim L(\la_i) < 2 (\dim G) / m$, there are few possibilities for $\la_i$.  If $G$ has type $C_2$, then the upper and lower bounds on $\dim L(\la_i)$ in the preceding two sentences are 4 and 5, so there are no possibilities.  If $G$ has type $A_\ell$, then we find only one possibility, that $G$ has type $A_3$ and $\la_i = \omega_2$ with $\dim L(\omega_2) = 6$.  The resulting representations appear in the third row of Table \ref{meta.few}.
\end{proof}

\subsection*{When $k[V]^G$ is a polynomial ring} 
Regarding cases where $\dim V/G$ is small, it is clear that $\dim V/G = 0$ if and only if there is a dense open $G$-orbit in $V$, if and only if $V/G = \Spec k = \Aff^0$.  Similarly, $\dim V/G \le 1$ if and only if there is a dense open $G$-orbit in $\P(V)$, see, for example, \cite[Prop.~12]{Popov:14} for $\car k = 0$ or \cite[Prop.~6.1]{BGL} for $\car k$ arbitrary.  We have: If $\dim V/G = 1$, then $V/G \cong \Aff^1$.

If $\dim V/G = 2$, $G$ simple, and $\car k = 0$, then $V/G \cong \Aff^2$ by \cite[Th.~2.4]{Kempf:some}.  Is the same conclusion true if $k$ is allowed to have prime characteristic?  We prove the following.

\begin{prop} \label{adj.cofree}
Let $V$ be a faithful irreducible representation of a simple algebraic group $G$.  If $\dim V/G = 2$, then $V/G \cong \Aff^2$ unless perhaps $(G, \car k, V)$ is $(\Spin_5, 5, L(\omega_1 + \omega_2))$ or $(\Spin_{13}$, any, spin).
\end{prop}

\begin{proof}
We apply the classification of possibilities for $(G, \car k, V)$ provided by Theorem \ref{kVG}.  By Corollary \ref{adj.inv}, we may assume that $V \ne L(\hst)$.  The minimum of $\dim k[V]^G$ for the representations in Table \ref{meta.few} is 5, so $(G, \car k, V)$ belongs to Table \ref{meta.small} or \ref{meta.big}.

For $(\PSp_6, \ne 3, L(\omega_2)$) and  $(\PSp_8$, 2, $L(\omega_2)$), the ring $k[V]^G$ is described in \cite[Examples 8.3, 8.5]{GG:simple} and it is polynomial.

The representations $(F_4, \ne 2, 3$, natural), $(\PGL_2, \ne 2, 3, S^4 k^2)$, and $(\PGL_3, \ne 2, 3, S^3 k^3)$ arise as $\theta$-groups (a.k.a.~Vinberg representations) where the overgroup is of type $E_6$, $A_2$, and $D_4$ and the automorphism $\theta$ of the overgroup has order 2, 2, and 3 respectively.  (See, e.g., \cite[p.~260--262]{PoV} or \cite[p.~1154]{RLYG}.)  By \cite[Th.~4.23]{Levy:Vinberg}, $k[V]^G$ is a polynomial ring.  

The case $(F_4$, 2, natural) is not covered by the result from \cite{Levy:Vinberg}.  Instead we refer to \cite[Example 11.4]{GG:simple}, which shows that $k[V]^G$ is a polynomial ring with generators of degree 2 and 3.
\end{proof}


\section{Regular orbits}  \label{sec:regular}

In this section we consider when a simple algebraic group $G$ acting on an irreducible
module $V$ has a regular orbit, i.e., when there is some $v \in V$ with $G_v = 1$.  We will 
consider this first just for the points and then consider the same problem for
group schemes, resulting in a proof Theorem \ref{t:regorbit}.

The following  is an immediate consequence of \cite[Table 1]{GurLawther}.  We only state the result in characteristics other than $2$, $3$, and $5$
as this is what we shall use. 

\begin{prop} \label{t:regorbitpoints}   Let $G$ be a simple algebraic group over
the algebraically closed field $k$ of characteristic $p := \car k \ne 2, 3, 5$ such that $G_v(k)$ acts faithfully
and irreducibly on $V$.     If $G_v(k)$ is finite for some $v$, then either
$G_w(k) =1 $ for some $w$ or  the following holds (up to twists by
Frobenius or graph automorphisms), 
$p \ne 0$,  $G$ is a quotient of $\SL_{\ell +1}$, and $V=L(\omega_1 + p^e \omega_1)$ or $L(\omega_1 + p^e \omega_\ell)$ for some $e \ge 1$.
\end{prop}

We now consider the case when there is a regular orbit for the group scheme.  Of course,
if $\car k=0$, then the previous result implies that there exists a regular orbit if and only
if the generic stabilizer is finite, if and only if $\dim V > \dim G$ (Theorem \ref{summ.thm}).

In the following, we write $P\Omega_n$ for $n = 3$ or $n \ge 5$ for the adjoint group of type $B_\ell$ (when $n = 2\ell+1$ is odd) or $D_\ell$ (when $n = 2 \ell$ is even).

\begin{lem}  \label{l:regorbitO}   
Let $G=P\Omega_n$ for $n =3$ or $n \ge 5$ with $\car k \ne 2$ and $V$ the nontrivial irreducible composition factor
of the symmetric square of the orthogonal module.  Then there exists $v \in V$ with
$G_v$ the trivial group scheme.
\end{lem}

\begin{proof}  Let $W$ denote the space of symmetric $n \times n$ matrices of trace $0$.
View $\SO_n(k)$ as the subgroup of $\SL_n(k)$  of matrices $A$ with $AA^{\top}=1$.  
Then $\SO_n(k)$ acts on $W$ by  conjugation with kernel the center.    Note that 
$W  = V$ unless $p$ divides $n$, in which case $V  = W/W_0$ where $W_0$ are the scalar matrices.

We recall that over an algebraically closed field, every matrix is similar to a symmetric matrix, see for example \cite[Lemma 3.1]{BGS}.   In particular, let $A$ be a symmetric nilpotent matrix with minimal polynomial
of degree $n$.     
The centralizer of $A$ in $M_n(k)$ is the subalgebra $k[A]$, which consists of symmetric matrices.
This shows that $\mathfrak{g}_A=0$, since $\mathfrak{g}$ is the Lie algebra of skew symmetric matrices. 
If $U$ is orthogonal and $UA=AU$, then  $U=f(A)$ for some polynomial $f$.  As $U$ is symmetric and orthogonal, it
is an involution.  There are no non-scalar involutions in $k[A]$ and so $G_A(k)=1$, whence
$G_A$ is the trivial group scheme.

This completes the proof if $p$ does not divide $n$.  If $p$ does divide $n$, the result follows
by observing that if any matrix commutes with $A$ modulo scalars, it commutes with $A$,
since the only nilpotent matrix in the set $A  + \lambda I$ is $A$.
\end{proof}

\begin{proof}[Proof of Theorem \ref{t:regorbit}]
Recall that $V$ is a faithful and irreducible representation of a simple group $G$ over a field $k$ of characteristic $\ne 2, 3, 5$. 

In case \eqref{reg.dim}, i.e., $\dim V \le \dim G$, then $\dim G_v > 0$ for generic $v \in V$ (Th.~\ref{summ.thm}), so there cannot be a regular orbit.  In case \eqref{reg.SL4}, $\dim \g_v > 0$ for generic $v \in V$, so $\dim \g_{v'} > 0$ for all $v' \in V$ and there cannot be a regular orbit.  In case \eqref{reg.usual}, there is no $v \in V$ with $G_v(k) = 1$ \cite[Prop.~5.1.8]{GurLawther}, so there cannot be a regular orbit.  In summary, if any of \eqref{reg.dim}, \eqref{reg.SL4}, or \eqref{reg.usual} hold, then the other three conditions fail.  

Therefore, we assume that $\dim V > \dim G$ and that we are not in case \eqref{reg.SL4} nor \eqref{reg.usual}, and we aim to prove the existence of a regular orbit.  We may assume the s.g.p.\ is not the trivial group scheme, whence by Theorem \ref{summ.thm}, up to graph automorphism $(G, \car k, V)$ belongs to Table \ref{meta.big}.  Lemma \ref{l:regorbitO} handles four rows of the table, including $(\SL_2 / \mu_2, \ne 2, 3, S^4 k^2)$ with $n = 3$ and $(\SL_4/\mu_4 \ne 2, L(2\omega_2))$ with $n = 6$.  The row $(\SL_2, \ne 2, 3, S^3 k^2)$ is from Example \ref{char3}.  

The remaining cases in Table \ref{meta.big} are examples of $\theta$-groups, where $G$ is the identity component of the subgroup of an overgroup $H$ fixed by an automorphism $\theta$ of finite order.  The cases for $(G, V)$ are:

\begin{enumerate}
\renewcommand{\theenumi}{\roman{enumi}}
\item \label{reg.sl3s3} $(\PGL_3, L(3\omega_1))$, with $H$ of type $D_4$ and $\theta$ of order 3; 
\item \label{reg.sl8w4}   $(\SL_8/\mu_4, L(\omega_4))$, with $H$ of type $E_7$ and $\theta$ of order 2;  
\item \label{reg.sl9}  $(\SL_9/\mu_3, L(\omega_3))$, with $H$ of type $E_8$ and $\theta$ of order 3;
\item  \label{reg.hspin16} $(\Spin_{16}/\mu_2, \text{half-spin})$, with $H$ of type $E_8$ and $\theta$ of order 2; or
\item \label{reg.psp8}  $(\PSp_8, L(\omega_4))$, with $H$ of type $E_6$ and $\theta$ of order 2.
\end{enumerate}

So  Let $G < H$ be as in the five cases above.  Since $p \ne 2,3$ or $5$,  $\car k$ is good for $H$.   
Then $G=C_H(\theta)^\circ$ and $\lsub$ is the direct sum of the eigenspaces of $\theta$.  Note that  $\g$ is the trivial
eigenspace and $V$ can be identified with the nontrivial eigenspace or one of the nontrivial eigenspaces if $\theta$
has order $3$ and in that case $V^*$ is the other eigenspace.    It follows by \cite[5.1.4]{GurLawther} that there exists 
a regular nilpotent element  $v \in V$ and $G_v(k)=1$.   It also follows by the computation in \cite{GurLawther} that 
the centralizer of $v$ in $H$  is contained in the sum of the nontrivial eigenspaces of $\theta$
whence its intersection with $\lsub$ is $0$.   Let $J=C_H(n)$.   So $J$ is an abelian group of dimension equal to the rank
of $H$ and $\theta$ normalizes $H$ with $C_J(\theta)=1$.  In good characteristic, the centralizer of $v$ in $\lsub$ is
the Lie algebra of $C_J(\theta)$ and this is $0$  since $\theta$ acts without fixed points on   $C_H(v)$ and so also on
its Lie algebra.  
\end{proof}

\section{Not-necessarily-semisimple representations} \label{section.sec}

Let $W$ be a section of a representation $V$ of an algebraic group $G$.  That is, there are $G$-invariant subspaces $V_1 \subseteq V_2 \subseteq V$ so that $W \cong V_2 / V_1$ as representations of $G$.  In this section, we discuss  connections between the stabilizers $G_w$ and $G_v$ of generic $w \in W$ and $v \in V$ respectively.

If $W$ is a summand of $V$, then one can take $w$ to be a projection of $v$ in $W$, in which case $G_w$ evidently contains $G_v$, compare \cite[Lemma 2.15]{Loetscher:edsep}.  Unfortunately, this statement does not easily extend to the case where $W$ is not a summand of $V$, see \cite[Example 2.6]{GG:large}, which gives an example with $G = \Ga$ where $W$ is a subspace of $V$ and $G_w = 1$, yet $G_v \ne 1$.  (See also Example \ref{char3} for a different but related phenomenon.)

We do know, by an easy argument using upper semicontinuity of dimension, that
\begin{equation} \label{usc.ineq}
\dim G_w \ge \dim G_v \quad \text{and} \quad \dim \g_w \ge \dim \g_v
\end{equation}
for generic $w \in W$ and $v \in V$ when $W$ is a section of $V$, see for example \cite[Example 2.2]{GG:large}.

Recall that a representation $V$ of $G$ is \emph{generically free} if $G_v = 1$ for generic $v \in V$.  We have: \emph{If $\car k = 0$ and a representation $V$ of $G$ has a section $W$ that is generically free, then $V$ is generically free} because $W$ is a summand of $V$.  Theorem \ref{sections} below provides a version of this in prime characteristic.

\subsection*{Separably free actions}
Note that if $V$ is generically free, the kernel $N$ of the action is necessarily trivial.  To accommodate the possibility that $N \ne 1$, we make the following definition.

\begin{defn}
A representation $V$ of $G$ is \emph{separably free} if the kernel $N := \ker  [ G \to \GL(V) ]$ of the action is \'etale and $(G/N)_v = 1$ for generic $v \in V$.
\end{defn}

\begin{thm} \label{sections} 
Let $V$ be a representation of a simple algebraic group $G$.  If $V$ has an irreducible section that is separably (resp.~generically) free as a representation of $G$, then $V$ is separably (resp.~generically) free.
\end{thm}

Before proving the theorem, we note some lemmas.  The following is well known, and a proof is contained in \cite[\S10]{GG:simple}, compare \cite[Lemma 2.6(i)]{GG:spin}.

\begin{lem} \label{GL.cor4.easy}
Let $V$ be a representation of a semisimple algebraic group $G$ over an algebraically closed field $k$.  If, for every $g \in G(k)$ that is (i) noncentral semisimple and whose image in $\GL(V)$ has prime order or (ii) unipotent, we have $\dim V^g + \dim g^G < \dim V$, then for generic $v \in V$, $G_v(k)$ is central in $G(k)$.
\end{lem}

As a consequence of previous work on irreducible representations of simple groups, we have the following converse:

\begin{lem}[Corollary 7 in \cite{GurLawther}] \label{GL.cor4}
Let $V$ be a faithful and irreducible representation of a simple algebraic group $G$ over an algebraically closed field $k$.  The stabilizer $G_v(k) = 1$ for generic $v \in V$ if and only if 
$\dim V^g + \dim g^G < \dim V$ for all $g \in G(k)$ of prime order (modulo
the center) and, if $\car k = 0$, all unipotent $g$.
\end{lem}

The analogue of Lemma \ref{GL.cor4} for Lie algebras is false.  For example, when $\car k = p \ne 0$, for any vector space $W$, the irreducible and faithful representation $V = W \otimes W^{[p]}$ of $G = \SL(W)$ has $\g_v = 0$ for generic $v \in V$, yet $\dim x^G + \dim V^x = \dim V + \dim W - 2$ for $x$ a root element, see \cite[\S10]{GG:large}.

\begin{proof}[Proof of Theorem \ref{sections}]
Suppose first that the irreducible section $W$ of $V$ is separably free.  Put $N := \ker [G \to \GL(W)]$ and $Z := \ker [G \to \GL(V)]$, so $Z \subseteq N$.  By hypothesis, $N$ is \'etale, so $Z$ is also.    In particular, both $N$ and $Z$ are central in $G$.

If any summand of $V$ is separably free, then $V$ is separably free.  Writing $V$ 
as a direct sum of the $Z(G)$-homogeneous components we may assume that $V$ is $Z(G)$-homogeneous.

By Lemma \ref{GL.cor4}, the inequality $\dim W^g + \dim g^G < \dim W$ holds for the relevant $g \in G(k)$.  Then it is easy to see that $\dim V^g + \dim g^G < \dim V$ for those same $g$.  The easier lemma, Lemma \ref{GL.cor4.easy}, now shows that $G_v(k)$ is a central subgroup for generic $v \in V$.  As $V$ is $Z(G)$-homogeneous, $G_v(k) = Z(k)$.

Because $\Lie(N) = \Lie(Z) = 0$, the natural maps $\Lie(G) \to \Lie(G/N)$ and $\Lie(G) \to \Lie(G/Z)$ are isomorphisms, and we obtain isomorphisms of generic stabilizers $\Lie(G/N)_w \cong \g_w$ and $\Lie(G/Z)_v \cong \g_v$.  As $\dim \g_v \le \dim \g_w = 0$, we find that $G_v$ is \'etale, so $G_v = Z$, i.e., $V$ is separably free.

In case $W$ is generically free, then (1) $N = 1$ so $Z = 1$ and (2) $V$ is separably free by the above.  So $V$ is generically free.
\end{proof}

\begin{conjecture}
If $G$ is reductive and $W$ is a generically
free section of $V$, then $V$ is generically free. 
\end{conjecture}

\section{Representations with the same invariants} \label{same.sec}

We give a new proof of one of the main results of \cite{GG:simple}, Theorem \ref{same.thm} below, which characterizes inclusions where the subgroup and overgroup have the same invariants.  The original proof relied on results from \cite{seitzmem}, whereas the following, quite different proof avoids that reference and instead uses the information about the dimension of the generic stabilizer.  This approach leaves very few cases to examine.  In this final section of the paper, we stop viewing algebraic groups as affine schemes and instead view them in a naive way as the group of $k$-points under the Zariski topology, as is done, for example in \cite{Hum:LAG}.

\begin{thm}[Theorem 13.1 in \cite{GG:simple}]  \label{same.thm}
Suppose that $G < H \le \SL(V)$ with $G$ a simple algebraic group over an algebraically closed field $k$ acting irreducibly on $V$, and $H$ connected and closed in $\SL(V)$.  If $\dim k[V]^G = \dim k[V]^H$, then $k[V]^G = k[V]^H$ and one of the following holds, up to a Frobenius twist and/or a twist by a graph automorphism:
\begin{enumerate}
\item \label{same.SL} $H = \SL(V)$ and $k[V]^G = k$ (i.e., $\dim k[V]^G = 0$).
\item \label{same.Sp} $H = \Sp(V)$, $\car k = 2$, $G = G_2$, $\dim V = 6$, and $k[V]^G = k$.
\item \label{same.SO} $H = \SO(V)$, $k[V]^G = k[q]$ for a homogeneous quadratic form $q$ (in particular, $\dim k[V]^G = 1$).
\item $(G, H, V, \car k)$ is in Table \ref{same.table}.  
\end{enumerate}
\end{thm}
The possibilities for $G$ in \eqref{same.SL} and \eqref{same.SO} can be extracted from Tables \ref{meta.small} and \ref{meta.big}; see Tables C and D in \cite{GG:simple} for an explicit list.

\begin{rmk}
The statement of Theorem \ref{same.thm} is slightly different from the one in \cite{GG:simple}.  The earlier version erroneously omitted the hypothesis that $H$ is connected.  Of course, the dimension of the ring of invariants only depends
on the connected component of the identity but certainly the actual ring of invariants can change.   

The earlier version also omitted the inclusion $G_2 < \PSp_6$ from Table \ref{same.table}.  For this inclusion, $G_2$ embeds in $\PSp_6$ when $\car k = 2$ via the natural irreducible representation of $G_2$.  The 14-dimensional representation $L(\omega_2)$ of $\PSp_6$ restricts to the adjoint representation of $G_2$.  To see this, note that the representation $\wedge^2 k^6$ of $\PSp_6$ is $k \oplus L(\omega_2)$,  and compare the restrictions of this and the adjoint module to an $A_2$ subgroup of $G_2$.  The fact that the ring of invariants on this representation is polynomial with generators of degrees 2 and 3 as in Table \ref{same.table} is \cite[Example 8.5]{GG:simple} for $\PSp_6$ and Example \ref{kg2} for $G_2$.
The equality $k[V]^G = k[V]^H$ in the other cases was proved in \cite{GG:simple}. 
\end{rmk}


\begin{table}[hbt]
\begin{center}
\begin{tabular}{@{}ccccc@{}} \toprule
$G$&$H$&$\dim V$&$\car k$&degrees \\ \midrule
$\PGL_3$&$G_2$&7&3&2 \\
$G_2$&$\PSp_6$&14&2&2, 3 \\
$\Spin_{11}$&$\HSpin_{12}$&32&all&$\begin{cases} \scriptstyle{4}&\scriptstyle{\text{if $\car k \ne 2$}}\\ \scriptstyle{2} & \scriptstyle{\text{if $\car k = 2$}}\end{cases}$ \\
$\SO_{2n}$ ($n \ge 3$)&$\PSp_{2n}$&$\begin{cases} \scriptstyle{2n^2 - n - 2}& \scriptstyle{\text{if $n$ even}}\\ \scriptstyle{2n^2 -n-1} & \scriptstyle{\text{if $n$ odd}} \end{cases}$&2&$\begin{cases}?\\2, 3, \ldots, n\end{cases}$ \\
$\SO_8$ or $\Sp_8$&$F_4$&26&2&2, 3\\
$\SL_n$ & $\SL_n \otimes \SL_n$ & $n^2$ & $ \ne 0$ & $n$ \\ \bottomrule
\end{tabular}
\medskip
\caption{Representations referred to in Theorem \ref{same.thm}, copied from Table E in \cite{GG:simple}.  The representations in the last row are denoted $L(\omega_1 + p^e  \omega_\ell)$ and $L(\omega_1 + p^e \omega_1)$ in Table \ref{meta.big}.  Note that since this section views algebraic groups in the naive sense, when $\car k = 2$ we have natural identifications $\SO_n = P\Omega_n$ and $\Sp_{2n} = \PSp_{2n}$.} \label{same.table}
\end{center}
\end{table}

We use the following.

\begin{lem} Let $G < H \le \SL(V)$ with $G$ a simple algebraic group
acting irreducibly on $V$ and $H$ a closed connected subgroup.
Then $\dim k[V]^H = \dim k[V]^G$ if and only $GH_v$ is dense in $H$
for generic $v \in V$.
\end{lem}

\begin{proof}    If $GH_v$ is dense in $H$,  then $Gv = GH_v v$ is dense in $Hv$
for generic $v$ and so the maximal dimension of $G$ and $H$ orbits are the
same.   Since the dimension of the ring of invariants is the codimension of
a maximal dimensional orbit, the density assumption implies that
the dimensions of the ring of invariants are the same.

Conversely, if the dimensions of the ring of invariants are the same, then
there exists a nonempty subset $\mathcal{O}$ of $V$ such that
$\dim Gv = \dim Hv$ for all $v \in \mathcal{O}$ and since $Hv$ is an
irreducible variety $Gv$ is dense in $Hv$.  Thus, $\dim GH_v = \dim G
+ \dim H_v - \dim G_v
= \dim H_v + \dim H - \dim H_v = \dim H$ and so $GH_v$ is dense in $H$.
\end{proof}

\begin{lem} \label{l:codimsp}
Let $G$ be a simple algebraic irreducible subgroup of $H=\Sp_{2\ell}$ with $\ell \ge 3$.  If $\dim G \ge \dim H - 3 \ell$,
then $\car k=2$ and either (1) $G = \SO_{2\ell}$ or (2) $G=G_2$ and $\ell = 3$.  
\end{lem}

\begin{proof}
The right side of the inequality in the statement is $2\ell^2 - 2\ell$, which is increasing for $\ell \ge 3$.  Its minimum value is 12, so $G$ cannot have type $A_1$, $A_2$, $B_2$, or $C_2$.  Otherwise, the natural representation of $H$ is an irreducible self-dual representation $L(\la)$ of $G$ of dimension between $6$ and $2\ell_G$, where $\ell_G$ is the real number $\ge 3$ such that $2\ell_G^2 - 2\ell_G = \dim G$.

In case $\la$ is restricted, the tables in \cite{luebeck} verify the claim.  Indeed, for other group types, the smallest self-dual irreducible module with restricted highest weight is already too large.

In general, we may write $\la = \la_0 + p\la_1$ with $\la_0$ restricted and $p := \car k \ne 0$.  Since $G$ acts faithfully on $L(\la)$, $\la_0 \ne 0$.  By hypothesis, $L(\la)$ is self-dual, i.e., $\la$ is fixed by $-w_0$ for $w_0$ the longest element of the Weyl group, so the same is true for $\la_0$, $\la_1$.  Then $\dim L(\la) = \dim L(\la_0) \cdot \dim L(\la_1)$, so the examination of the dimensions of self-dual representations in the preceding paragraph shows that $\la_1 = 0$, i.e., $\la$ is restricted.
\end{proof} 

\begin{proof} [Proof of Theorem \ref{same.thm}]
In view of the remarks just after the statement of the theorem, it suffices to show that
equality of the dimension of the ring of invariants only occurs in the cases listed in the conclusion.

We first consider the case that $H$ is not simple.   If not, then
$V$ is tensor decomposable for $H$ and so also $G$ whence by \cite{GurLawther},
$G_v$ is generically finite (and almost always generically trivial by \cite{GG:irred}).
In particular, $\dim k[V]^G = \dim V - \dim G$.  

We may assume that $H = H_1 \times H_2$, $G$ embeds in $H_i$
by the projection $\pi_i$ and that $V=V_1 \otimes V_2$ where $V_i$ is an irreducible
 $H_i$-module.   Let $J:=\pi_1(H) \times \pi_2(H) \cong G \times G$.  So $G \le J \le H$.
  
 Let $d_i = \dim V_i$ and assume that $d= d_1 \le d_2$.  
  Let  $v = \sum_{i=1}^d e_i \otimes f_i$ where the $e_i$ constitute a basis for $V_1$.
 Observe that $\pi_2$ restricted to $H_v$ has trivial kernel.  Indeed, if $h \in H_v$ and $\pi_2(h)=1$, 
 then $h$ fixes each $f_i$ and so fixes $v$ if and only if $he_i=e_i$ for all $i$, whence 
 $h$ is trivial on $V_1$ and so on $V$.  
 
   If $d_1 < d_2$, then we see that $J_v \cong \pi_2(H_v)$ stabilizes the span of $f_1, \ldots, f_d$
   and so $\pi_2(H_v)$ is properly contained in $H_2$.  Thus $\dim J_v < \dim G$ and 
   so $\dim k[V]^H \le \dim k[V]^J  <  \dim V - \dim H + \dim G = \dim V - \dim G = \dim k[V]^G$.   
 
 So we may assume that $d_1=d_2$ and identify $V_1 = V_2 = W$ (as vector spaces rather than $G$-modules).
   Note that in $L = \SL(W) \otimes \SL(W)$, the stabilizer in $L$ of 
 a generic vector  is a diagonal subgroup $D$, i.e., it is isomorphic to $\SL(W)$ and the projection onto either factor
 is a bijection.  
 If $\pi_2(G) \ne \SL(W)$,  then clearly $\dim \pi_2(G \cap D) < \dim G $ for generic $D$.  Thus, $\dim J_v < \dim G$
 for generic $v$ and so arguing as above, we see that $\dim k[V]^H < \dim k[V]^G$.  
 If $\pi_2(G) = \SL(W)$,  then $G \cong \SL(W)$ and $V$ is a tensor product of two $G$-modules each of which
 are Frobenius twists of the natural module or its dual.  In this case,  $\dim k[V]^G=1$ and $(G,H,V)$ are as in the last line of Table \ref{same.table}.
 
\medskip

So now assume that $H$ is simple. If $\dim H_v=0$ for some $v \in V$,
then $\dim k[V]^H = \dim V - \dim H < \dim V -\dim G = \dim k[V]^G$, a contradiction.
So $\dim V \le \dim H$,  $(H, \car k, V)$ is in Table \ref{meta.small} or $V$ is the irreducible part of the adjoint module for $H$.

 First consider the case that $V$ is the irreducible part of the adjoint module for $H$. 
 As $H$ acts faithfully, $\car k$ is not special and $H$ is adjoint.  Lemma \ref{adj.sub} provides a contradiction.

Next consider the case that $H=\Sp(W)= \Sp_{2\ell}$ for some $\ell > 2$ and $V= L(\omega_2)$.  
Thus, $\dim V = \ell(2\ell -1) - \delta$ with $\delta =1$ or $2$.   Suppose that $G$ is not irreducible
on $W$.   Let $X$  be a $G$-composition factor of $W$ of maximal dimension $e$.   
 So $e \le 2\ell - 2$.   If $e = 2$, then the largest composition factor of $G$ on $\wedge^2 W$ is $4$-dimensional,
 a contradiction.   Otherwise, 
  $G$ has a nontrivial composition factor on $V$ of dimension at most $\dim \wedge^2X =  (\ell -2)(2\ell - 3) < \dim V$,
whence $G$ is not irreducible on $V$, a contradiction. 
 
  Also,  $GH_v$ is dense in $H$
 so $\dim G \ge \dim H - 3 \ell$, since $H_v$ is generically of dimension $3\ell$.   Now
apply Lemma \ref{l:codimsp} to deduce that $(G, H, V, \car k)$ are as in the 2nd or 4th lines of Table \ref{same.table}.
 
In the remaining cases from Table \ref{meta.small}, we have $d= \dim k[V]^H \le 2$.   If
$d=1$, then inspection of Tables \ref{meta.small}, \ref{meta.big}, and \ref{meta.few} show \eqref{same.SL}, \eqref{same.Sp}, or \eqref{same.SO} of the theorem hold or 
we are in the case of line 1 of Table \ref{same.table}.   So we may assume that $d =2$.  

 First assume that 
 $\dim V >  \dim G$  and so $d:= \dim k[V]^G = \dim V - \dim G > 0$.
Then by Theorem \ref{kVG}
one of the following holds: 
\begin{itemize}
\item  $G=A_1$, $ p \ne 2,3$, $V = L(\omega_4)$ with $\dim V = 5$; 
\item  $G=A_2$, $p \ne 2,3$, $V = L(3\omega_1)$ with $\dim V =10$; or
\item  $G=B_2$, $p=5$,  $V=L(\omega_1 + \omega_2)$ with $\dim V =12$.
\end{itemize}
  The possibilities for
$H$ with $\dim k[V]^H = 2$ are also given in Theorem \ref{kVG} and we see that
there are no examples.

Next consider the case that $\dim V \le \dim G$.   If $V$ is the nontrivial composition
factor of the adjoint module (and $\car k$ is not special for $G$) and $\dim k[V]^G \le 2$,
then $(G, p)$ appears in Table \ref{adj.small.table}.
Again, the possibilities for $H$ are all given in the tables and we see the only
examples are captured in the 2nd and 4th rows of Table \ref{same.table}.

Finally assume that $\dim V < \dim G$ and and we are not in the case
of the adjoint module.  Thus, $G$ and $H$ both occur in Table \ref{meta.small} and we
see that there are no containments (when $d=2$). 
\end{proof}

An immediate consequence of Theorem \ref{same.thm} is:

\begin{cor}
Under the hypotheses of Theorem \ref{same.thm}: If $G_v$ is finite for generic $v \in V$
or $\dim V > \dim G$, 
then $\car k \ne 0$ and $(G, H, V, \car k)$ are as in the last row of Table \ref{same.table}.
 \end{cor}

\begin{cor}  Let $V$ be a faithful and irreducible representation of a simple algebraic group $G$.
Let $e$ be the greatest common divisor of the degrees of the homogeneous elements of $k[V]^G$.  
Assume that $(G,V)$ is not given (up to twist) in the statement of Theorem \ref{same.thm}.   If $m$ is a sufficiently
large multiple of $e$, then for almost all homogeneous $f \in k[V]^G$ of degree $m$,  $G$ is the identity
component of the stabilizer of $f$.
\end{cor}

\begin{proof}  There are only finitely many closed subgroups $H$ of $\SL(V)$ containing $G$
(see \cite{LiebeckTest} or  \cite[Prop.~9.2]{GG:simple}) and by the previous result $\dim k[V]^H < \dim k[V]^G$.
Thus, for $m$ sufficiently large (and a multiple of $e$),  the set of $f \in k[V]^G$ that
are homogeneous whose
stabilizer has connected component strictly containing $G$ is a finite union of proper subspaces
of the degree $m$ invariants of $G$, whence the result.
\end{proof}

In fact using results of Seitz and Testerman and others (see \cite{seitzmem}, \cite{BurnGMT, BurnGT, BurnMT}), for most
$(G,V)$ it is the case that $G$ is maximal in the corresponding classical group ($\SO(V), \Sp(V)$,
or $\SL(V)$) and so any homogeneous  $G$-invariant (other than a scalar times a power of
the invariant quadratic form in the case $G < \SO(V)$) has stabilizer whose connected component
is $G$.   See \cite{GG:simple} for many examples of this, e.g., if $G=E_8$ and $V$ is the adjoint module.


\providecommand{\bysame}{\leavevmode\hbox to3em{\hrulefill}\thinspace}
\providecommand{\MR}{\relax\ifhmode\unskip\space\fi MR }
\providecommand{\MRhref}[2]{%
  \href{http://www.ams.org/mathscinet-getitem?mr=#1}{#2}
}
\providecommand{\href}[2]{#2}

\end{document}